\documentclass{amsart}

\usepackage{amssymb}
\usepackage{amsmath}
\usepackage[pdftex]{graphicx} 
\usepackage{tikz}
\usetikzlibrary{matrix,calc}
\usepackage{enumerate}
\usepackage{url}
\usepackage{rotating,makecell}
\usepackage{multirow}
\usepackage{enumitem}

\newtheorem{theorem}{Theorem}[section]
\newtheorem{lemma}[theorem]{Lemma}
\newtheorem{corollary}[theorem]{Corollary}
\newtheorem{proposition}[theorem]{Proposition}

\newtheorem{conjecture}[theorem]{Conjecture}
\newtheorem{example}[theorem]{Example}
\newtheorem{remark}[theorem]{Remark}
\newtheorem{definition}[theorem]{Definition}

\newtheorem{conditions}[theorem]{Conditions}
\newtheorem{notation}[theorem]{Notation}

\numberwithin{equation}{section}

\newcommand{\HH}{\mathcal{H}}
\newcommand{\EE}{\mathcal{E}}
\newcommand{\OO}{\mathcal{O}}
\newcommand{\PZ}{\mathcal{PZ}}
\newcommand{\ZZ}{\mathcal{Z}}

\newcommand{\hprime}[1]{
  {\renewcommand{\arraystretch}{0.55}\renewcommand{\arraycolsep}{0.7pt}
  \begin{array}{ccc}
  #1
  \end{array}}
}

\newcommand{\hsq}{
  \multicolumn{2}{c}{\multirow{2}{*}{\large$\square$}}
}

\begin{document}

\newenvironment{smallarray}[1]
  {\gdef\MatName{#1}\begin{tikzpicture}
    \matrix[
    matrix of math nodes,
    row sep=-\pgflinewidth,
    column sep=-\pgflinewidth,
    nodes={inner sep=1pt,rectangle,text width=2mm,align=center},
    text depth=0mm,
    text height=2mm,
    nodes in empty cells
    ]  (#1)}
  {\end{tikzpicture}}
\def\MyZ(#1,#2){%
  \draw ([xshift=-3.6pt,yshift=3.6pt] $ (\MatName-#1-#2)!0.5!(\MatName-\the\numexpr#1+1\relax-\the\numexpr#2+1\relax) $ ) rectangle ([xshift=3.6pt,yshift=-3.6pt] $ (\MatName-#1-#2)!0.5!(\MatName-\the\numexpr#1+1\relax-\the\numexpr#2+1\relax) $ );
}

\title[Primes of Quantum $SL_3$ and Poisson-primes of $SL_3$]%
{The Prime Spectrum of Quantum $SL_3$ and the Poisson-prime Spectrum of its Semi-classical Limit}
\author{Si\^an Fryer}
\address{Department of Mathematics, University of California at Santa Barbara, California 93106, USA.} 
\email{sianfryer@math.ucsb.edu}


\maketitle

\begin{abstract}
A bijection $\psi$ is defined between the prime spectrum of quantum $SL_3$ and the Poisson prime spectrum of $SL_3$, and we verify that $\psi$ and $\psi^{-1}$ both preserve inclusions of primes, i.e. that $\psi$ is in fact a homeomorphism between these two spaces.  This is accomplished by developing a Poisson analogue of Brown and Goodearl's framework for describing the Zariski topology of spectra of quantum algebras, and then verifying directly that in the case of $SL_3$ these give rise to identical pictures on both the quantum and Poisson sides.  As part of this analysis, we study the Poisson primitive spectrum of $\mathcal{O}(SL_3)$ and obtain explicit generating sets for all of the Poisson primitive ideals.
\end{abstract}

\section{Introduction}

In the study of noncommutative quantum algebras, the theory of $\HH$-stratification (due originally to Goodearl and Letzter \cite{Hprimes2}) allows us to partition the prime spectra of various algebras into strata indexed by the prime ideals which are invariant under the rational action of some algebraic torus $\HH$ (hence the name).  The prime ideals in an individual stratum can be far more easily understood -- each stratum is homeomorphic to the prime spectrum of a commutative Laurent polynomial ring -- and under fairly mild conditions on the algebra and torus action we can even pinpoint the primitive ideals as those prime ideals which are maximal in their strata.  

A corresponding theory for Poisson algebras (usually appearing as the semi-classical limits of various quantum algebras; informally, this means we let $q \rightarrow 1$) has also been developed in \cite{Goodearl_Poisson}; we recall the main results in Section~\ref{sec:background} below.  For more information on semi-classical limits and the relationship between Poisson and quantum algebras see for example \cite{GoodearlSummary}; the definitive references for $\HH$-stratification are \cite{GBbook} for quantum algebras and \cite{Goodearl_Poisson} for Poisson algebras.

Many tools developed originally for quantum algebras (e.g. the $\HH$-stratification mentioned above, and Cauchon's deleting derivations algorithm) turn out to transfer naturally to the Poisson case with very little modification.  These two pictures -- prime ideals in quantum algebras $A$ for $q \in k^{\times}$ not a root of unity, and Poisson prime ideals in their semi-classical limits $R$ -- are often remarkably alike, and there have been many recent results exploring these similarities, for example \cite{GLL2,LaunoisLecoutre1,MR3185525}. In light of this, Goodearl makes the following conjecture in \cite{GoodearlSummary}:

\begin{conjecture}\label{conj:goodearl}
\cite[Conjecture~9.1]{GoodearlSummary} Assume that $k$ is algebraically closed of characteristic zero, and let $A$ be a generic quantized coordinate ring of an affine algebraic variety $V$ over $k$.  Then $A$ should be a member of a flat family of $k$-algebras with semiclassical limit $\OO(V)$, and there should be compatible homeomorphisms $spec(A) \longrightarrow pspec(R)$ and $prim(A) \longrightarrow pprim(R)$, where $pspec$ and $pprim$ denote the spaces of Poisson prime and Poisson primitive ideals respectively.
\end{conjecture}

This conjecture has been verified for some types of quantum algebras, e.g. $q$-commuting polynomial rings and quantum affine toric varieties \cite[Theorem~4.2]{GLz2}, multiparameter quantized symplectic and Euclidean spaces \cite{MR2416732,MR2724222}, and $\OO_q(SL_2)$ \cite[Example~9.7]{GoodearlSummary}.  In addition, it is known that there are bijections $prim(A) \longrightarrow pprim(R)$ for many algebras of the form $A=\OO_q(G)$ (under various conditions on $G$ and $q$, see e.g. Joseph \cite{MR1297159}, Hodges, Levasseur, Toro \cite{MR1440253}, Yakimov \cite{MR3185525}) but the topological properties of these maps remain unknown.

The problem is as follows: while we have an excellent understanding of the individual strata in $spec(A)$ or $pspec(R)$, which is sufficient to obtain the bijections mentioned above, this tells us nothing about the interaction of primes from different strata.  A homeomorphism $\psi: spec(A) \longrightarrow pspec(R)$ and its inverse $\psi^{-1}$ must preserve inclusions between primes; if $A$ and $R$ are both noetherian and $char(k)=0$, the converse is also true \cite[Lemma~9.4]{GoodearlSummary}.  In $\OO_q(SL_2)$ the picture is simple enough that we can check these inclusions directly, but for larger algebras this quickly becomes impossible.

In \cite{GBrown}, Brown and Goodearl develop a conjectural framework to tackle this problem in the case of quantum algebras, by encoding the information about inclusions of primes from different strata in terms of certain commutative algebras and maps between them.  This commutative data can then (in theory) be computed, providing a full picture of the topology of the space $spec(A)$.  Verifying that this framework applies to a given algebra reduces to a question of normal generation of prime ideals modulo their $\HH$-prime, which is currently unknown in general but is accomplished in \cite[\S5-7]{GBrown} for the cases of $\OO_q(GL_2)$, $\OO_q(M_2)$ and $\OO_q(SL_3)$.

In Section~\ref{sec:framework} below we develop the corresponding framework for the Poisson case, and obtain a very similar picture (up to the obvious modifications of replacing ``centre'' with ``Poisson centre'', ``normal'' with ``Poisson normal'', etc).  In particular, since the quantum and Poisson pictures each reduce to understanding a finite number of commutative algebras and associated maps (commutative with respect to both the associative multiplication and the Poisson bracket), this suggests a possible approach to tackling Conjecture~\ref{conj:goodearl}: describe these commutative algebras, and the homeomorphism will follow.  As proof of concept, we carry out this program in the case of $\OO_q(SL_3)$ and $\OO_q(GL_2)$ in Sections \ref{sec:SL3 poisson stuff} and \ref{sec:homeomorphism}.

In Section~\ref{sec:SL3 poisson stuff} we verify that the results of Section~\ref{sec:framework} apply to the Poisson algebras $\OO(SL_3)$ and $\OO(GL_2)$ (with respect to the Poisson bracket induced from the corresponding quantum algebras); these results follow naturally from the examination of Poisson prime and Poisson primitive ideals in $\OO(SL_3)$ which formed part of the author's PhD thesis \cite{MeThesis}.  In Section~\ref{sec:homeomorphism} we prove the following main theorem:

\begin{theorem}\label{res:theorem homeomorphism intro} (Theorem~\ref{res:theorem homeomorphism}, Corollary~\ref{res:homeomorphism sl3 primitives theorem})
Let $k$ be an algebraically closed field of characteristic zero and $q \in k^{\times}$ not a root of unity.  Then there is a bijection $\psi: spec(\OO_q(SL_3)) \longrightarrow pspec(\OO(SL_3))$ which is a homeomorphism with respect to the Zariski topologies on each space, and which restricts to a homeomorphism $prim(\OO_q(SL_3)) \longrightarrow pprim(\OO(SL_3))$.
\end{theorem}

\begin{corollary} (Corollary~\ref{res:homeomorphism theorem gl2})
Let $k$, $q$ be as above.  There is a homeomorphism $\psi': spec(\OO_q(GL_2)) \longrightarrow pspec(\OO(GL_2))$, which restricts to a homeomorphism $prim(\OO_q(GL_2))\longrightarrow pprim(\OO(GL_2))$.
\end{corollary}

While the direct computational approach used here will not be feasible for algebras of larger dimension, it does demonstrate how the techniques of \cite{GBrown} and this paper for ``patching together'' the topologies of individual strata in well-behaved quantum and Poisson algebras could be used to tackle Conjecture~\ref{conj:goodearl} in greater generality.

\section{Background and Definitions}\label{sec:background}

Throughout, fix $k$ to be an algebraically closed field of characteristic 0.  

We will restrict our attention exclusively to commutative Poisson algebras in this paper.  If $R$ is a Poisson algebra, an ideal $I$ is called a \textit{Poisson ideal} if it is also an ideal with respect to the Poisson bracket, i.e. $\{I,R\} \subseteq I$; it is called \textit{Poisson prime} if it is a Poisson ideal which is prime in the usual commutative sense\footnote{This is equivalent to the standard definition of a Poisson-prime ideal when $R$ is noetherian and $k$ has characteristic 0; all algebras considered here will be of this type.  See \cite[Lemma 1.1(d)]{Goodearl_Poisson}.}, and \textit{Poisson primitive} if it is the largest Poisson ideal contained in some maximal ideal of $R$.  We write $pspec(R)$ for the set of Poisson primes in $R$, and $pprim(R)$ for the set of Poisson primitives.

The space $pspec(R)$ naturally inherits the Zariski topology of $spec(R)$, and in turn induces a topology on $pprim(R)$.  The closed sets of these topologies are defined to be:
\[V(I):= \{P \in pspec(R): P \supseteq I\}, \qquad V_p(I):= V(P) \cap pprim(R),\]
where we may assume that $I$ is a \textit{Poisson} ideal in $R$ (if it is not, simply replace it with the intersection of all of the ideals in $V(I)$; this is a Poisson ideal that defines the same closed set as $V(I)$).  When $R$ is noetherian, the sets $\{V(P): P \in pspec(R)\}$ form a basis for this topology: this follows from the fact that $I = \bigcap_{i=1}^nP_i$, where $I$ is a Poisson ideal and $P_1, \dots, P_n$ are the finitely many Poisson prime ideals minimal over $I$ (see \cite[Lemma 1.1]{Goodearl_Poisson}).

For any topological space $T$, we will denote the set of its closed sets by $CL(T)$.

An element $z \in R$ is called \textit{Poisson central} if $\{z,a\} = 0$ for all $a \in R$; define the \textit{Poisson centre} of $R$ to be
\[\PZ(R) = \Big\{z \in R: \{z,R\} = 0\Big\}.\]
Similarly, we say that $r \in R$ is \textit{Poisson normal} if $\{r,R\} \subseteq rR$.

If $I$ is a Poisson ideal, then there is an induced Poisson structure on $R/I$ given by 
\begin{equation}\label{eq:poisson bracket on quotient}\{a + I,b+I\} = \{a,b\} + I.\end{equation}
In a similar vein, if $R$ is a domain and $X$ is a multiplicative set, then the Poisson bracket on $R$ extends uniquely to a Poisson bracket on $R[X^{-1}]$ via the formula
\begin{equation}\label{eq:extension of bracket to localization}
\{ax^{-1},by^{-1}\} = \{a,b\}x^{-1}y^{-1} - \{a,y\}bx^{-1}y^{-2} - \{x,b\}ax^{-2}y^{-1} + \{x,y\}abx^{-2}y^{-2}.
\end{equation}
The Poisson centre of $R[X^{-1}]$ can be related to $R$ as follows: 
\begin{equation}\label{eq:poisson centre of localization}\PZ(R[X^{-1}]) = \big\{ax^{-1} \in R[X^{-1}] : \{ax^{-1},R\} = 0\big\}.\end{equation}
The $\subseteq$ direction of this equality is clear, while the $\supseteq$ direction follows from \eqref{eq:extension of bracket to localization}.

The Poisson structures in which we are interested arise naturally as the semi-classical limits of various quantized coordinate rings.  However, since we will quickly narrow our attention to specific examples, we will not need the formal details of this relationship here and the interested reader is referred to \cite{GoodearlSummary}.  In this context, however, it makes sense to impose the following conditions on our Poisson algebras.

\begin{conditions}\label{conditions}\ 

\begin{enumerate}
\item $R$ is a commutative affine $k$-algebra with a Poisson bracket.
\item $\HH = (k^{\times})^r$ is an algebraic torus acting rationally on $R$ by Poisson automorphisms (see \cite[\S2]{Goodearl_Poisson}).
\item $R$ has only finitely many Poisson prime ideals which are invariant under the action of $\HH$.
\end{enumerate}
\end{conditions}

Denote by $\HH$-$pspec(R)$ the set of prime ideals in $R$ which are simultaneously Poisson ideals and invariant under the action of $\HH$; by \cite[Lemma 3.1]{Goodearl_Poisson}, these are precisely the Poisson $\HH$-prime ideals.  Although we will often refer to these ideals simply as ``$\HH$-primes'', in the context of Poisson algebras we will always mean the \textit{Poisson} $\HH$-primes only.

\begin{remark}\label{rem:homogeneous elements}
Under our assumptions on $\HH$ and $k$, rational $\HH$-actions on $R$ correspond to gradings by the rational character group $X(H) \cong \mathbb{Z}^r$, with $\HH$-eigenvectors corresponding to non-zero homogeneous elements (see e.g. \cite[II.2]{GBbook}, \cite[\S2]{Goodearl_Poisson}).  This will be the only grading we consider here, and the terms ``eigenvectors'' and ``homogeneous elements'' will be used interchangably. 
\end{remark}

The $\HH$-primes induce partitions on $pspec(R)$ and $pprim(R)$ as follows: for $J \in \HH$-$pspec(R)$, define
\begin{gather*}
pspec_J(R) = \left\{P \in pspec(R) : \bigcap_{h \in \HH} h(P) = J\right\};\\
pprim_J(R) = pspec_J(R) \cap pprim(R).
\end{gather*}
For many algebras (in particular, those satisfying Conditions~\ref{conditions}) the topology of the individual strata can be understood using the following theorem:

\begin{theorem}[Poisson Stratification Theorem and Dixmier-Moeglin Equivalence]\label{res:poisson stratification} \cite[Theorems 4.2, 4.3, 4.4]{Goodearl_Poisson} Let $R$ be a commutative noetherian Poisson algebra, $k$ an algebraically closed field of characteristic 0, and $\HH = (k^{\times})^r$ an algebraic torus acting rationally on $R$ by Poisson automorphisms.  For $J \in \HH$-$pspec(R)$, we have:
\begin{enumerate}
\item The localization of $R/J$ at the set $\EE_J$ of homogeneous non-zero elements in $R/J$ is again a Poisson algebra upon which $\HH$ acts rationally, and $R_J:= R/J[\EE_J^{-1}]$ is a Poisson $\HH$-simple ring, i.e. admits no non-zero Poisson $\HH$-primes.  The Poisson centre $\PZ(R_J)$ is a Laurent polynomial ring over $k$ in at most $r$ indeterminates.
\item $pspec_J(R)$ is homeomorphic to $pspec(R_J)$ via localization and contraction, and $pspec(R_J)$ is in turn homeomorphic to $spec(\PZ(R_J))$ via contraction and extension.
\item If $\HH$-$pspec(R)$ is finite and $R$ is an affine $k$-algebra, then the Poisson primitive ideals in $pprim_J(R)$ correspond to maximal ideals in $\PZ(R_J)$, and the homeomorphisms above restrict to homeomorphisms $pprim_J(R) \approx pprim(R_J) \approx max(\PZ(R_J))$.
\end{enumerate}
\end{theorem}

While this allows us to understand each stratum individually, it carries no information about interactions between different strata.  In order to understand the topology of the Poisson spectrum as a whole, some additional tools will be required.

\begin{definition}\label{def:finite stratification}
Let $T$ be a topological space and $\Pi$ a poset.  A partition $T = \bigsqcup_{i \in \Pi} S_i$ is called a \textit{finite stratification} of $T$ if $\Pi$ is finite and
\begin{enumerate}
\item Each $S_i$ is nonempty and locally closed in $T$;
\item $\overline{S_i} = \bigsqcup_{j \in \Pi, \ j \geq i} S_j$, where $\overline{S_i}$ denotes the closure of $S_i$ in $T$.
\end{enumerate}
\end{definition}

If $T = \bigsqcup_{i \in\Pi}S_i$ is a finite stratification of a topological space $T$ then we can equip each $S_i$ with the induced topology from $T$ and, for any pair $i$, $j \in \Pi$, define a map $CL(S_i)\rightarrow CL(S_j)$ by 
\begin{equation}\label{eq:definition of varphi}\varphi_{ij}(Y) = \overline{Y}\cap S_j.\end{equation}
(Here $\overline{Y}$ denotes the closure of $Y$ in $T$.)  By property (ii) of Definition~\ref{def:finite stratification} above this map will only be interesting when $i < j$, so we will often focus only on this case.  Given a finite stratification of a topological space $T$, and its associated family of maps $\{\varphi_{ij}: i, j \in \Pi,\ i < j\}$, we can deduce some information about the structure of $T$ as follows:

\begin{lemma}\label{res:properties of finite stratification} \cite[Lemma 2.3]{GBrown}
Let $T$ be a topological space with a finite stratification $T = \bigsqcup_{i \in \Pi} S_i$, and let $\{\varphi_{ij} : i < j \in \Pi\}$ be the collection of maps defined in \eqref{eq:definition of varphi}.
\begin{enumerate}
\item Each $\varphi_{ij}$ satisfies $\varphi_{ij}(\emptyset) = \emptyset$ and $\varphi_{ij}(S_i) = S_j$.
\item Each $\varphi_{ij}$ preserves finite unions.
\item A subset $X \subseteq T$ is closed in $T$ if and only if
\begin{enumerate}
\item $X \cap S_i \in CL(S_i)$ for all $i \in \Pi$; and
\item $\varphi_{ij}(X\cap S_i) \subseteq X \cap S_j$ for all $i < j \in \Pi$.
\end{enumerate}
\end{enumerate}
\end{lemma}

Understanding the topology of the spaces $S_i$ and the action of the maps $\varphi_{ij}$ is enough to completely describe the topology of $T$.  The aim of Section~\ref{sec:framework} is to give (in certain particularly nice settings) an alternative definition of the $\varphi_{ij}$ which does not rely on already knowing the topology of $T$ itself.  First, we look at how the spaces $pspec(R)$ and $pprim(R)$ fit into this finite stratification framework:

\begin{lemma}\label{res:pspec has a finite stratification}
Let $R$ be a commutative Poisson algebra satisfying Conditions~\ref{conditions}.  The partitions
\[pspec(R) =  \mkern-27mu\bigsqcup_{J \in \HH\textrm{-}pspec(R)} \mkern-27mu pspec_J(R) \qquad \textrm{ and }\qquad pprim(R) = \mkern-27mu\bigsqcup_{J \in \HH\textrm{-}pspec(R)}\mkern-27mu pprim_J(R)\]
form finite stratifications of the topological spaces $pspec(R)$ and $pprim(R)$ respectively.
\end{lemma}

\begin{proof}
Under the assumptions of Conditions~\ref{conditions}, $\HH$-$pspec(R)$ is a finite set, and we can view it as a poset with respect to inclusion of ideals.  By standard properties of the $\HH$-strata and the Zariski topology on $pspec(R)$,  for each $J \in \HH$-$pspec(R)$ we have:
\begin{gather*}
J \in pspec_J(R),\\
pspec_J(R) = V(J) \Big\backslash \Big(\bigsqcup_{\substack{K \in \HH\textrm{-}pspec(R), \\ K \supsetneqq J}} \mkern-27mu V(K)\Big),\\
\overline{pspec_J(R)} = V(J) = \mkern-27mu \bigsqcup_{\substack{K \in \HH\textrm{-}pspec(R),\\ K \supseteq J}} \mkern-27mu pspec_K(R);
\end{gather*}
i.e. each $pspec_J(R)$ is nonempty, locally closed, and satisfies the closure condition given in Definition~\ref{def:finite stratification}.  The claim for $pprim(R)$ follows similarly, using the fact that $R$ is affine and so $J$ is precisely the intersection of the Poisson primitive ideals in $pprim_J(R)$ (e.g. \cite[Lemma 1.1]{Goodearl_Poisson}).
\end{proof}

If $R$ is a Poisson algebra as in Lemma~\ref{res:pspec has a finite stratification} and $P \in pspec_J(R)$ for some Poisson $\HH$-prime $J$, it follows that
\[V(P) = \mkern-18mu \bigsqcup_{\substack{K \in \HH\textrm{-}pspec,\\ K \supseteq J}} \mkern-18mu \varphi_{JK}\Big(V(P) \cap pspec_J(R)\Big);\]
this is because the $\subseteq$ in Lemma~\ref{res:properties of finite stratification}$(3)(ii)$ becomes an equality when $X = \overline{\{x\}}$ for some $x \in S_i$.

In other words, in order to understand inclusions of primes in $pspec(R)$ it is enough to find an equivalent definition of the $\varphi_{JK}$ which does not depend on already knowing the topology of $pspec(R)$.  In \cite{GBrown} Brown and Goodearl examine the corresponding problem for non-commutative quantum algebras $A$, which admit a similar stratification by $\HH$-primes; the interested reader is referred to \cite[II.1-2]{GBbook} for the background and definitions in this setting.  They develop a framework for understanding the action of the maps $\varphi_{JK}$ in the setting of $A$ in terms of certain commutative ``intermediate'' algebras $\ZZ_{JK}$ (see \cite[Definition~3.8]{GBrown}) and maps between them of the form
\[f_{JK}: \ZZ_{JK} \longrightarrow \ZZ(A_K), \qquad g_{JK}: \ZZ_{JK} \longrightarrow \ZZ(A_J),\]
where $A_J$ denotes the localization of $A/J$ at the set of all homogeneous $\HH$-eigenvectors as before, and $\ZZ(-)$ denotes the centre.  The maps $f_{JK}$ and $g_{JK}$ are homomorphisms of commutative $k$-algebras, and so they induce morphisms of schemes
\[f_{JK}^{\circ}: spec(\ZZ(A_K)) \longrightarrow spec(\ZZ_{JK}), \qquad g_{JK}^{\circ}: spec(\ZZ(A_J)) \longrightarrow spec(\ZZ_{JK}).\]
In \cite{GBrown}, Brown and Goodearl make the following general conjecture:

\begin{conjecture}\label{conj:GB conjecture}\cite[Conjecture 3.11]{GBrown} Let $A$ be a quantum algebra satisfying \cite[Assumptions~3.1]{GBrown}, let $J \subset K$ be $\HH$-primes, and let $f_{JK}$ and $g_{JK}$ be defined as in \cite[\S3]{GBrown}.  Then
\begin{equation}\label{eq:statement of conjecture in quantum case}\varphi_{JK} = f_{JK}^{\circ}\overline{|}g_{JK}^{\circ},\end{equation}
where
\begin{equation}\label{eq:definition of overline notation}f_{JK}^{\circ}\overline{|}g_{JK}^{\circ}(Y) := f_{JK}^{\circ \ -1}\left(\overline{g_{JK}^{\circ}(Y)}\right), \quad Y \in CL\big(spec(\ZZ(A_J))\big).\end{equation}
\end{conjecture}

\begin{proposition}\cite[Proposition~4.2]{GBrown}
Let $A$ be a quantum algebra satisfying \cite[Assumptions~3.1]{GBrown}.  Suppose that for each $\HH$-prime $J$ in $\HH$-$spec(A)$ and for every $P \in spec_J(A)$, $P/J$ has a generating set consisting of normal elements in $A/J$.  Then Conjecture~\ref{conj:GB conjecture} holds in $A$.
\end{proposition}

While it is not immediately apparent how to check this condition in general, it certainly holds in low-dimensional examples which have been computed directly; in particular, it is verified for the algebras $\OO_q(GL_2)$, $\OO_q(M_2)$ and $\OO_q(SL_3)$ in \cite[\S5-7]{GBrown}.

Our first aim is to develop the corresponding Poisson version of this framework; we will see that it is very similar to the quantum case, and indeed many of the proofs go through with almost no modification.  By making it possible to reduce both the topologies of $spec(A)$ for quantum $A$ and $pspec(R)$ for $R$ the semi-classical limit of $A$ to questions of commutative algebra, we will then be able to compare the two topologies directly.

\section{A framework for patching together the strata of $pspec(R)$}\label{sec:framework}
In this section our goal is to construct intermediate commutative algebras $\PZ_{JK}$ for each pair of $\HH$-primes $J \subset K$ in Poisson algebras satisfying Conditions~\ref{conditions}, and also certain homomorphisms between these algebras and the Poisson centres $\PZ(R_J)$, $\PZ(R_K)$.  The results in this section closely follow the corresponding quantum results in \cite{GBrown}.

Throughout, let $R$ be a commutative Poisson algebra satisfying Conditions~\ref{conditions}, and let $J \subset K$ be Poisson $\HH$-primes in $R$.

Recall from Theorem~\ref{res:poisson stratification} the definitions
\begin{gather*}\EE_J := \{\textrm{all homogeneous non-zero elements in }R/J\}, \\ R_J := R/J[\EE_J^{-1}]. \end{gather*}
By Theorem~\ref{res:poisson stratification}, the topological spaces $pspec_J(R)$ and $spec(\PZ(R_J))$ are homeomorphic via the map $P \mapsto PR_J \cap \PZ(R_J)$; we will often identify these two spaces without comment. 

Define a new multiplicative set by
\begin{equation}\label{eq:E_JK}\EE_{JK} := \{c \in R/J: c \textrm{ is homogeneous in }R/J\textrm{ and is not in }K/J\}.\end{equation}
Unhampered by the problems of non-commutative localization faced in \cite{GBrown}, we can immediately define
\begin{equation}\label{eq:PZ}\PZ_{JK} := \PZ(R/J[\EE_{JK}^{-1}]),\end{equation}
The Poisson bracket on $R/J[\EE_{JK}^{-1}]$ is the unique extension of the one on $R$ via the formulas given in \eqref{eq:poisson bracket on quotient} and \eqref{eq:extension of bracket to localization}.  It is clear that \eqref{eq:PZ} is equivalent to (a Poisson version of) the definition used in \cite{GBrown}.  

Since $\EE_{JK} \subset \EE_J$, the natural inclusion map $R/J[\EE_{JK}^{-1}] \hookrightarrow R_J$ restricts to a map on the Poisson centres.  This is a homomorphism of commutative $k$-algebras and we will make use of it below, so we name it as follows:
\begin{equation}\label{eq:map g}
g_{JK}: \PZ_{JK} \hookrightarrow \PZ(R_J).
\end{equation}

We also require a map $\PZ_{JK} \longrightarrow \PZ(R_K)$; this is the purpose of the next lemma.

\begin{lemma}\label{res:definition of f}Let $\pi_{JK}: R/J \longrightarrow R/K$ denote the natural projection map.  Then the map $f_{JK}: \PZ_{JK} \longrightarrow \PZ(R_K)$ defined by $f_{JK}(ab^{-1}) = \pi_{JK}(a)\pi_{JK}(b)^{-1}$ is a homomorphism of commutative algebras.
\end{lemma}

\begin{proof}
Let $\widetilde{f}$ be the composition of Poisson homomorphisms $\widetilde{f}: R/J \rightarrow R/K[\EE_K^{-1}]: a \mapsto \pi_{JK}(a)1^{-1}$; this extends uniquely to a Poisson homomorphism $R/J[\EE_{JK}^{-1}] \rightarrow R/K[\EE_{K}^{-1}]$ (which we also denote by $\widetilde{f}$) since $\widetilde{f}(\EE_{JK}) \subseteq \EE_K$.

Since projection to a quotient algebra sends Poisson-central elements to Poisson-central elements, it is easy to check that the following restriction is well defined:
\begin{equation}\label{eq:map f}
\begin{gathered}f_{JK}:= \widetilde{f}|_{\PZ_{JK}}: \PZ_{JK} \rightarrow \PZ(R_K),\\
ab^{-1} \mapsto \pi_{JK}(a)\pi_{JK}(b)^{-1},
\end{gathered}\end{equation}
and this is exactly the homomorphism of commutative $k$-algebras promised in the statement of the lemma.
\end{proof}

By rewriting $\PZ_{JK}$ in the same form as \cite[Definition 3.6]{GBrown}, it is an easy exercise to check that this definition of $f_{JK}$ agrees with (the Poisson equivalent of) the map defined in \cite[Lemma 3.10]{GBrown}.

We are now ready to put everything together.  Recall that whenever $h: R \rightarrow S$ is a homomorphism of commutative algebras, $h^{\circ}$ will denote the induced comorphism $spec(S) \rightarrow spec(R)$.

\begin{definition}\label{def:exciting map of interest}
Let $J \subset K$ be Poisson $\HH$-primes, and let $\PZ_{JK}$, $g_{JK}$ and $f_{JK}$ be as defined in \eqref{eq:PZ}, \eqref{eq:map g} and \eqref{eq:map f} respectively.  Define a map $CL(spec(\PZ(R_J))) \rightarrow CL(spec(\PZ(R_K)))$ as follows:
\[f_{JK}^{\circ} \overline{|} g_{JK}^{\circ} (Y) := f_{JK}^{\circ \ -1}\Big(\overline{g_{JK}^{\circ}(Y)}\Big), \quad Y \in CL(spec(\PZ(R_J))).\]
\end{definition}

Corresponding to \cite[Conjecture 3.11]{GBrown}, we expect that for all pairs of $\HH$-primes $J \subset K$ in Poisson algebras $R$ satisfying Conditions~\ref{conditions} we will have 
\begin{equation}\label{eq:varphi=fg}\varphi_{JK} =f_{JK}^{\circ} \overline{|} g_{JK}^{\circ}.\end{equation}
(Recall that we are identifying $pspec_J(R) \approx spec(\PZ(R_J))$ via the homeomorphism from Theorem~\ref{res:poisson stratification}.)

In order to establish conditions under which \eqref{eq:varphi=fg} holds, we need to describe the action of both $\varphi_{JK}$ and $f_{JK}^{\circ}\overline{|}g_{JK}^{\circ}$ on closed sets.  Since most of the action takes place in purely commutative algebras, the proofs necessarily follow those of \cite[\S4]{GBrown} closely.

We begin by restricting to the case $J = 0$ and describing the action of $f_{0K}^{\circ}\overline{|}g_{0K}^{\circ}$ on closed sets of the form $V(P) \cap pspec_0(R)$ for some $P \in pspec_0(R)$.  Since the next few proofs are quite technical, we hope that this approach will serve to illustrate the main ideas before we state the result in full generality (Proposition~\ref{res:main equality of maps result} below).

\begin{lemma}\label{res:first description of f|g action}
Set $J=0$, let $K$ be any $\HH$-prime, and suppose $P \in pspec_0(R)$.  Define $g_{0K}$ and $f_{0K}$ as in \eqref{eq:map g},\eqref{eq:map f} respectively, and let $Y = V(P) \cap pspec_0(R)$.  Then
\[f_{0K}^{\circ}\overline{|}g_{0K}^{\circ}(Y) = \{Q \in pspec_K(R) : f(PR_0 \cap \PZ_{0K}) \subseteq QR_K\}.\]
\end{lemma}

\begin{proof}
As always, we identify $pspec_0(R)$ and $spec(\PZ(R_0))$ via the homeomorphism from Theorem~\ref{res:poisson stratification}.  To simplify the notation we will drop the subscripts from maps and write $\varphi$, $\pi$, $f$ and $g$ for $\varphi_{0K}$, $\pi_{0K}$, $f_{0K}$ and $g_{0K}$ respectively.

First, consider the set $\overline{g^{\circ}(Y)}$: since $g$ is the natural embedding map $\PZ_{0K} \hookrightarrow \PZ(R_0)$, the comorphism $g^{\circ}$ is given by contraction to $\PZ_{0K}$.  Contraction clearly preserves inclusions and so $g^{\circ}(Y)$ contains a minimal element, namely $(PR_0\cap \PZ(R_0)) \cap \PZ_{0K} = PR_0 \cap \PZ_{0K}$.  We conclude that
\[\overline{g^{\circ}(Y)} = \{T \in spec(\PZ_{0K}) : T \supseteq (PR_0 \cap \PZ_{0K})\}. \]
For $Q \in pspec_K(R)$, we can now see that
\begin{align*}
Q \in f^{\circ}\overline{|}g^{\circ}(Y) &\iff f^{\circ}(QR_K \cap \PZ(R_K)) \in \overline{g^{\circ}(Y)}\\
&\iff f^{-1}(QR_K \cap \PZ(R_K)) \supseteq PR_J \cap \PZ_{0K} \\
&\iff QR_K \cap \PZ(R_K) \supseteq f(PR_0 \cap \PZ_{0K}).
\end{align*}
Finally, since $im(f) \subseteq \PZ(R_K)$ it is clear that $f(PR_0 \cap \PZ_{0K})$ is contained in $QR_K \cap \PZ(R_K)$ if and only if it is contained in $QR_K$, as required.
\end{proof}

Recall that we defined $\PZ_{JK} := \PZ(R/J[\EE_{JK}^{-1}])$, and that $f_{JK}$ is induced from the projection map $\pi_{JK}: R/J \rightarrow R/K$.  This allows us to rewrite $f(PR_0 \cap \PZ_{0K})$ as
\[f(PR_0 \cap \PZ_{0K}) = \{\pi_{JK}(p)\pi_{JK}(c)^{-1}: p \in P, \ c \in \EE_{0K}, \ pc^{-1} \in \PZ_{0K}\}.\]
As in \cite[Proposition 4.1]{GBrown}, we now observe that for $p \in P$ and $c \in \EE_{0K}$, we have 
\begin{align*}
\pi_{JK}(p)\pi_{JK}(c)^{-1} \in QR_K &\iff \pi_{JK}(p) \in QR_K \\
&\iff \pi_{JK}(p) \in QR_K \cap (R/K) = Q/K \\
&\iff p \in Q.
\end{align*}
We conclude that 
\[f(PR_0\cap \PZ_{0K}) \subseteq QR_K \iff \{p \in P : pc^{-1} \in \PZ_{0K} \textrm{ for some }c \in \EE_{0K}\} \subseteq Q.\]
This motivates the following result, as in \cite{GBrown}.

\begin{proposition}\label{res:condition for maps equal when J=0}
Set $J = 0$, let $K$ be any Poisson $\HH$-prime, and define $\PZ_{0K} \cdot \EE_{0K} = \{zc : z \in \PZ_{0K}, \ c \in \EE_{0K}\}$.  For $P \in pspec_0(R)$ and $Y = V(P) \cap pspec_0(R)$, we have
\begin{equation}\label{eq:characterisation of f|g}f_{0K}\overline{|}g_{0K}(Y) = \{Q \in pspec_K(R) : P \cap (\PZ_{0K} \cdot \EE_{0K}) \subseteq Q\}\end{equation}
and hence
\begin{equation}\label{eq:thing to prove}
\begin{array}{ccc}\multirow{2}{*}{$\varphi_{0K}(Y) = f_{0K}\overline{|}g_{0K}(Y)$}& $\textrm{ if and }$ &
P \cap (\PZ_{0K} \cdot \EE_{0K}) \subseteq Q \implies P \subseteq Q, \\ 
& $\textrm{ only if }$& \forall Q \in pspec_K(R).\end{array}\end{equation}


\end{proposition}

\begin{proof}
The equality \eqref{eq:characterisation of f|g} is clear from Lemma~\ref{res:first description of f|g action} and following discussion.  To prove \eqref{eq:thing to prove}, first recall that since $P \in pspec_0(R)$ the action of $\varphi_{0K}$ on $Y$ is as follows:
\begin{align*}\varphi_{0K}(Y) &= V(P) \cap pspec_K(R). \\
&= \{Q \in pspec_K(R) : Q \supseteq P\},\end{align*}
Combining this with \eqref{eq:characterisation of f|g}, the inclusion $\varphi_{0K}(Y) \subseteq f_{0K}\overline{|}g_{0K}(Y)$ is clear; meanwhile, the reverse inclusion holds if and only if
\[P \cap (\PZ_{0K} \cdot \EE_{0K}) \subseteq Q \implies P \subseteq Q, \quad \forall Q \in pspec_K(R).\]
\end{proof}

With all of the pieces in place, we can now expand our attention to arbitrary pairs of $\HH$-primes $J \subset K$ and to arbitrary closed sets $Y \in CL(pspec_J(R))$.  The next proposition shows that the general case follows easily from the specific cases considered above.

\begin{proposition}\label{res:main equality of maps result}
Let $R$ be a commutative Poisson algebra satisfying Conditions~\ref{conditions}, let $J \subset K$ be $\HH$-primes, and define $g_{JK}$, $f_{JK}$ as in \eqref{eq:map g}, \eqref{eq:map f} respectively.  Then the equality $\varphi_{JK} = f_{JK}^{\circ}\overline{|}g_{JK}^{\circ}$ holds for all closed sets $Y \in CL(pspec_J(R))$ if and only if
\begin{equation}\label{eq:main condition for equality of the two maps}(P/J) \cap \PZ_{JK} \cdot \EE_{JK} \subseteq Q/J \implies P \subseteq Q\end{equation}
for all $P \in pspec_J(R)$ and all $Q \in pspec_K(R)$.
\end{proposition}

\begin{proof}
We begin by considering an arbitrary closed set $Y \in CL(pspec_J(R))$; this has the form $Y = V(I) \cap pspec_J(R)$ for some Poisson ideal $I$.  As in the quantum case, we can first observe that for any Poisson $\HH$-prime $L \supseteq J$ and any $Q \in V(I) \cap pspec_L(R)$ we have $J \subseteq Q$; hence
\[V(I) \cap pspec_L(R) = V(I+J) \cap pspec_L(R),\]
i.e. we may reduce to the case $J = 0$ without loss of generality.

Next, recall that $V(I) \cap pspec_J(R)$ is a finite union of sets of the form $V(P_i) \cap pspec_J(R)$, where the $P_i$ are the Poisson primes minimal over $I$.  Both of the maps $\varphi_{JK}$ and $f_{JK}^{\circ}\overline{|}g_{JK}^{\circ}$ preserve finite unions, so we can focus our attention on closed sets of the form $V(P) \cap spec_J(R)$ for $P \in pspec(R)$.  Since we have reduced to the case $J = 0$ we know that $J \subseteq P$, but we need to consider the possibility that there is some $L \in \HH$-$pspec(R)$ such that $J \subsetneq L \subseteq P$.  In this case, however, $V(P) \cap pspec_J(R) = \emptyset$, and it follows trivially that $\varphi_{JK} = f_{JK}^{\circ}\overline{|}g_{JK}^{\circ}$ on $V(P) \cap pspec_J(R)$.

The proposition therefore holds if and only if it holds in the case $J = 0$ and $Y = V(P) \cap pspec_J(R)$ for any prime $P$ in $pspec_J(R)$.  The result now follows from Proposition~\ref{res:condition for maps equal when J=0}.
\end{proof}

\begin{proposition}\label{res:conjecture satisfied if normal generation of primes}
Apply the same hypotheses as Proposition \ref{res:main equality of maps result}, and fix $P \in pspec_J(R)$.  If $P/J$ has a generating set consisting of Poisson-normal elements in $R/J$, then \eqref{eq:main condition for equality of the two maps} holds for that choice of Poisson prime $P$.
\end{proposition}

\begin{proof}
Fix some $Q \in pspec_K(R)$.  Since $J \subset Q$, we may as well assume $J = 0$.  We prove the contrapositive of \eqref{eq:main condition for equality of the two maps}, so suppose $P \not\subseteq Q$.  By assumption there is some non-zero element $f \in P \backslash Q$ which is Poisson-normal in $R$, which we can decompose as $f = f_1 + \dots + f_n$ for some $\HH$-eigenvectors $\{f_1, \dots, f_n\}$ with distinct eigenvalues.  It is standard that each $f_i$ will also be Poisson-normal in $R$, and further that they are compatible with $f$ in the following sense: for any $r \in R$, let $r_f \in R$ be such that $\{f,r\} = fr_f$, then we must have $\{f_i,r\} = f_ir_f$ as well.  (This follows from the corresponding result for homogeneous $r$, which can easily be seen by expanding out $\{f,r\} = \{f_1 + \dots + f_n,r\}$.)

Since $f \in P \backslash Q$, there is some $i$ such that $f_i \not\in Q$.  Now $f_i$ is regular, homogeneous and regular mod $K$, i.e. $f_i \in \EE_{JK}$.  Observe that in $R[\EE_{JK}^{-1}]$,
\[\{ff_i^{-1},r\} = \{f,r\}f_i^{-1} - \{f_i,r\}ff_i^{-2} = fr_ff_i^{-1} - f_ir_fff_i^{-2} = 0, \quad \forall r \in R,\]
so that $ff_i^{-1} \in \PZ_{JK}$ by \eqref{eq:poisson centre of localization}. We have shown that $f = (ff_i^{-1})f_i \in P \cap (\PZ_{JK} \cdot \EE_{JK})$ but $f \not\in Q$, as required.
\end{proof}

Finally, when applying these results we will often want to replace the set $\EE_{JK}$ with a subset $\widetilde{\EE}_{JK}$ such that $\PZ_{JK} = \PZ(R/J[\widetilde{\EE}_{JK}^{-1}])$; the conditions under which this can be done are described in the next lemma.

\begin{lemma}\label{res:shrinking E_JK set}
Let $\EE_{JK}$ be as defined in \eqref{eq:E_JK}, and suppose $\widetilde{\EE}_{JK} \subset \EE_{JK}$ is a multiplicative set which satisfies the following condition: for all $\HH$-primes $L$ such that $J \subset L$ and $L \not\subseteq K$, we have
\[L/J \cap \widetilde{\EE}_{JK} \neq \emptyset.\]
Then $\PZ_{JK} = \PZ(R/J[\widetilde{\EE}_{JK}^{-1}])$.
\end{lemma}

\begin{proof}
As above, it is enough to consider $J = 0$.  

One direction is clear: if $ac^{-1} \in \PZ(R[\widetilde{\EE}_{JK}^{-1}])$ then in particular $ac^{-1}$ Poisson-commutes with every element of $R$.  It now follows from \eqref{eq:poisson centre of localization} that $ac^{-1} \in \PZ_{JK}$ as well.

Conversely, suppose that $ac^{-1} \in \PZ_{JK}$; we need to show that we can find $a' \in R$, $c' \in \widetilde{\EE}_{JK}$ such that $ac^{-1} = a'c'^{-1}$.  Let $I$ be the Poisson ideal generated by $c$ in $R$: this is the smallest ideal in $R$ which contains $c$ and satisfies $\{I,R\} \subseteq I$.  We first show that $I \cap \widetilde{\EE}_{JK} \neq \emptyset$.  Suppose not: then the extension $I[\widetilde{\EE}_{JK}^{-1}]$ is a non-trivial Poisson ideal in $R[\widetilde{\EE}_{JK}^{-1}]$.  Let $P$ be any minimal Poisson prime ideal over $I[\EE_{JK}^{-1}]$, and
\[L:=\bigcap_{h \in \HH} h(P)\]
its Poisson $\HH$-prime.  $L$ is non-zero (in particular $c \in L$) so by standard localization theory it corresponds to a non-zero $\HH$-prime $L'$ in $R$ which is disjoint from $\widetilde{\EE}_{JK}$.  By the definition of $\widetilde{\EE}_{JK}$ this means $L' \subset K$, but now $c \in L' \subset K$ is a contradiction since $c$ is regular mod $K$.  We conclude that $I \cap \widetilde{\EE}_{JK} \neq \emptyset$.

As a commutative ideal, $I$ is generated by $c$ and all $n$-fold Poisson brackets $\{\dots\{c,b_1\},b_2\}\dots,b_n\}$, $n \geq 1$, $b_1, \dots, b_n \in R$ homogeneous.  In addition, it follows from the Poisson-centrality of $ac^{-1}$ that $\{a,b\}\{c,b\}^{-1} = ac^{-1}$ for any homogeneous $b \in R$, and hence inductively
\[ac^{-1} = \big\{\dots\{a,b_1\},b_2\}\dots,b_n\big\}\big\{\dots\{c,b_1\},b_2\},\dots,b_n\big\}^{-1},\]
with $n$, $b_1, \dots, b_n$ as above.  In particular, we can conclude from this that $ac^{-1}\cdot I \subset R$.

Let $c' \in I \cap \widetilde{\EE}_{JK}$, and let $a':=ac^{-1}c' \in R$.  But now we are done, since $a'c'^{-1} = ac^{-1}$ with $a' \in R$, $c' \in \widetilde{\EE}_{JK}$ as required.\end{proof}

\section{The Case of $\OO(SL_3)$}\label{sec:SL3 poisson stuff}

The results of the previous section give us a way to understand the topological structure of $pspec(R)$ in terms of certain commutative algebras and maps between them.  On the other hand, we have yet to see any examples of Poisson algebras which fit into this framework.  With this in mind, we now turn our attention to certain specific Poisson matrix varieties, and verify that they do indeed satisfy the conditions of Proposition~\ref{res:conjecture satisfied if normal generation of primes}.

We begin by defining $\OO(M_2)$, the coordinate ring of all $2 \times 2$ matrices with entries in $k$.  As a commutative algebra this is just the polynomial ring $k[a,b,c,d]$; viewed as the semi-classical limit of the corresponding quantum algebra $\OO_q(M_2)$ (defined in \S\ref{sec:homeomorphism} below) it acquires the structure of a Poisson algebra, with Poisson bracket defined by
\begin{equation}\label{eq:2x2 Poisson relations}\begin{gathered}
\{a,b\} = ab, \qquad \{a,c\} = ac \\
\{b,d\} = bd, \qquad \{c,d\} = cd \\
\{b,c\} = 0,\\
\{a,d\} = 2bc.
\end{gathered}\end{equation}
For more details about the semi-classical limit process, and in particular to see that \eqref{eq:2x2 Poisson relations} does indeed define a Poisson bracket on $\OO(M_2)$, see \cite[\S2]{GoodearlSummary}.

We can now scale this up to define the matrix Poisson variety $\OO(M_{m,p})$: this is the commutative polynomial ring on $mp$ variables $\{Y_{ij}: 1 \leq i \leq m, \ 1 \leq j \leq p\}$, where the Poisson bracket is defined by requiring that all ``rectangles'' of four variables $\{Y_{ij}, Y_{il}, Y_{kj}, Y_{kl}\}$ satisfy a copy of the relations \eqref{eq:2x2 Poisson relations}.  That this is a well-defined Poisson bracket on $\OO(M_{m,p})$ follows again from the general theory of semi-classical limits, e.g. \cite[\S2]{GoodearlSummary}.

If $I \subset \{1, \dots, m\}$ and $J \subset \{1, \dots, p\}$ with $|I| = |J|$, we write $[I|J]$ for the minor with rows in $I$ and columns in $J$.  The determinant in $\OO(M_m)$ will be denoted by $D$; this is always a Poisson-central element (this follows easily from, for example, the Poisson commutation relations for minors in \cite[\S5.1.1]{MeThesis}), and this allows us to define
\[\OO(SL_n) = \OO(M_n)/(D-1), \qquad \OO(GL_n) = \OO(M_n)[D^{-1}],\]
where the Poisson structure in each case is induced by the Poisson bracket on $\OO(M_m)$ as in \eqref{eq:poisson bracket on quotient}, \eqref{eq:extension of bracket to localization}.

As one might expect, all of the Poisson algebras defined above admit rational actions by tori $\HH$ acting by Poisson automorphisms.  In the case of $\OO(M_{m,p})$, the action is given as follows: 
\begin{equation}\label{eq:H action on M and GL}h = (\alpha_1, \dots, \alpha_m, \beta_1, \dots, \beta_p) \in (k^{\times})^{m+p},\qquad  h.Y_{ij} = \alpha_i\beta_jY_{ij}.\end{equation}
When $m=p$, this extends to a rational action on $\OO(GL_m)$.  Meanwhile, for $\OO(SL_m)$ we must instead consider the torus
\begin{equation}\label{eq:torus acting on SL}\{(\alpha_1, \dots, \alpha_m, \beta_1, \dots, \beta_m) \in (k^{\times})^{2m} : \alpha_1\alpha_2\dots\beta_m = 1\} \cong (k^{\times})^{2m-1},\end{equation}
which acts on $\OO(SL_m)$ via the action induced from \eqref{eq:H action on M and GL}.

Finally, each of the algebras $\OO(M_m)$, $\OO(GL_m)$ and $\OO(SL_m)$ admits a transpose automorphism, namely
\[\tau: Y_{ij} \mapsto Y_{ji}\]
which satisfies $\tau(D) = D$.  This is an automorphism of Poisson algebras, and provides a useful tool for reducing the number of computations we need to do when verifying results case-by-case.

\begin{figure}[b]
\setlength{\tabcolsep}{2pt}

\centering
\begin{tabular}{cc|cccccc}

& $\omega_{-}$ & \multirow{2}{*}{321} & \multirow{2}{*}{231} & \multirow{2}{*}{312} & \multirow{2}{*}{132}& \multirow{2}{*}{213}& \multirow{2}{*}{123} \\ 
$\omega_{+}$& &&&&&& \\ \hline 

\multicolumn{2}{c|} {\raisebox{1em}{321}} &

\begin{smallarray}{m321321}
{
 \circ & \circ & \circ \\
\circ & \circ & \circ \\
\circ & \circ & \circ \\
};
\end{smallarray}
&

\begin{smallarray}{m321231}
{
 \circ & & \\
\circ & & \\
\circ & \circ & \circ \\
};
\MyZ(1,2)
\end{smallarray}

&
\begin{smallarray}{m321312}
{
 \circ & \circ & \bullet \\
\circ & \circ & \circ \\
\circ & \circ & \circ \\
};
\end{smallarray}
&

\begin{smallarray}{m321132}
{
 \circ & \bullet & \bullet \\
\circ & \circ & \circ \\
\circ & \circ & \circ \\
};
\end{smallarray}
&

\begin{smallarray}{m321213}
{
 \circ & \circ & \bullet \\
\circ & \circ & \bullet \\
\circ & \circ & \circ \\
};
\end{smallarray}

&
\begin{smallarray}{m321123}
{
 \circ & \bullet & \bullet \\
\circ & \circ & \bullet \\
\circ & \circ & \circ \\
};
\end{smallarray}
\\

\multicolumn{2}{c|}{\raisebox{1em}{231}}
&
\begin{smallarray}{m231321}
{
 \circ & \circ & \circ \\
\circ & \circ & \circ \\
\bullet & \circ & \circ \\
};
\end{smallarray}

&
\begin{smallarray}{m231231}
{
 \circ &  &  \\
\circ &  &  \\
\bullet & \circ & \circ \\
};
\MyZ(1,2)
\end{smallarray}

&
\begin{smallarray}{m231312}
{
 \circ & \circ & \bullet \\
\circ & \circ & \circ \\
\bullet & \circ & \circ \\
};
\end{smallarray}
&

\begin{smallarray}{m231132}
{
 \circ & \bullet & \bullet \\
\circ & \circ & \circ \\
\bullet & \circ & \circ \\
};
\end{smallarray}

&
\begin{smallarray}{m231213}
{
 \circ & \circ & \bullet \\
\circ & \circ & \bullet \\
\bullet & \circ & \circ \\
};
\end{smallarray}

&
\begin{smallarray}{m231123}
{
 \circ & \bullet & \bullet \\
\circ & \circ & \bullet \\
\bullet & \circ & \circ \\
};
\end{smallarray}

\\

\multicolumn{2}{c|}{\raisebox{1em}{312}} 
&
\begin{smallarray}{m312321}
{
 \circ & \circ & \circ \\
 &  & \circ \\
 &  & \circ \\
};
\MyZ(2,1)
\end{smallarray}

&
\begin{smallarray}{m312231}
{
 \circ &  &  \\
 &  &  \\
 &  & \circ \\
};
\MyZ(1,2)
\MyZ(2,1)
\end{smallarray}

&
\begin{smallarray}{m312312}
{
 \circ & \circ & \bullet \\
 &  & \circ \\
 &  & \circ \\
};
\MyZ(2,1)
\end{smallarray}

&
\begin{smallarray}{m312132}
{
 \circ & \bullet & \bullet \\
 &  & \circ \\
 &  & \circ \\
};
\MyZ(2,1)
\end{smallarray}

&
\begin{smallarray}{m312213}
{
 \circ & \circ & \bullet \\
 &  & \bullet \\
 &  & \circ \\
};
\MyZ(2,1)
\end{smallarray}

&
\begin{smallarray}{m312123}
{
 \circ & \bullet & \bullet \\
 &  & \bullet \\
 &  & \circ \\
};
\MyZ(2,1)
\end{smallarray}

\\

\multicolumn{2}{c|}{\raisebox{1em}{132}} 

&
\begin{smallarray}{m132321}
{
 \circ & \circ & \circ \\
 \bullet & \circ & \circ \\
 \bullet & \circ & \circ \\
};
\end{smallarray}
&
\begin{smallarray}{m132231}
{
 \circ &  &  \\
 \bullet &  &  \\
 \bullet & \circ & \circ \\
};
\MyZ(1,2)
\end{smallarray}

&
\begin{smallarray}{m132312}
{
 \circ & \circ & \bullet \\
 \bullet & \circ & \circ \\
 \bullet & \circ & \circ \\
};
\end{smallarray}
&
\begin{smallarray}{m132132}
{
 \circ & \bullet & \bullet \\
 \bullet & \circ & \circ \\
 \bullet & \circ & \circ \\
};
\end{smallarray}
&
\begin{smallarray}{m132213}
{
 \circ & \circ & \bullet \\
 \bullet & \circ & \bullet \\
 \bullet & \circ & \circ \\
};
\end{smallarray}
&
\begin{smallarray}{m132123}
{
 \circ & \bullet & \bullet \\
 \bullet & \circ & \bullet \\
 \bullet & \circ & \circ \\
};
\end{smallarray}
\\

\multicolumn{2}{c|}{\raisebox{1em}{213}} 

&

\begin{smallarray}{m213321}
{
 \circ & \circ & \circ \\
 \circ & \circ & \circ \\
 \bullet & \bullet & \circ \\
};
\end{smallarray}
&
\begin{smallarray}{m213231}
{
 \circ &  &  \\
 \circ &  &  \\
 \bullet & \bullet & \circ \\
};
\MyZ(1,2)
\end{smallarray}

&
\begin{smallarray}{m213312}
{
 \circ & \circ & \bullet \\
 \circ & \circ & \circ \\
 \bullet & \bullet & \circ \\
};
\end{smallarray}
&
\begin{smallarray}{m213132}
{
 \circ & \bullet & \bullet \\
 \circ & \circ & \circ \\
 \bullet & \bullet & \circ \\
};
\end{smallarray}
&
\begin{smallarray}{m213213}
{
 \circ & \circ & \bullet \\
 \circ & \circ & \bullet \\
 \bullet & \bullet & \circ \\
};
\end{smallarray}
&
\begin{smallarray}{m213123}
{
 \circ & \bullet & \bullet \\
 \circ & \circ & \bullet \\
 \bullet & \bullet & \circ \\
};
\end{smallarray}
\\

\multicolumn{2}{c|}{\raisebox{1em}{123}}
&
\begin{smallarray}{m123321}
{
 \circ & \circ & \circ \\
 \bullet & \circ & \circ \\
 \bullet & \bullet & \circ \\
};
\end{smallarray}
&
\begin{smallarray}{m123231}
{
 \circ & & \\
 \bullet &  &  \\
 \bullet & \bullet & \circ \\
};
\MyZ(1,2)
\end{smallarray}

&
\begin{smallarray}{m123312}
{
 \circ & \circ & \bullet \\
 \bullet & \circ & \circ \\
 \bullet & \bullet & \circ \\
};
\end{smallarray}
&
\begin{smallarray}{m123132}
{
 \circ & \bullet & \bullet \\
 \bullet & \circ & \circ \\
 \bullet & \bullet & \circ \\
};
\end{smallarray}
&
\begin{smallarray}{m123213}
{
 \circ & \circ & \bullet \\
 \bullet & \circ & \bullet \\
 \bullet & \bullet & \circ \\
};
\end{smallarray}
&
\begin{smallarray}{m123123}
{
 \circ & \bullet & \bullet \\
 \bullet & \circ & \bullet \\
 \bullet & \bullet & \circ \\
};
\end{smallarray}

\end{tabular}

\caption{Generators for $\HH$-primes in $\OO(SL_3)$.}\label{fig:H_primes_gens}
\end{figure}

We now restrict our attention to the algebra $\OO(SL_3)$.  Recall that a $\HH$-prime of $\OO(SL_3)$ is a Poisson-prime ideal which is stable under the action of $\HH = (k^{\times})^{3}$ as defined in \eqref{eq:H action on M and GL}, \eqref{eq:torus acting on SL} above.  Explicit generating sets are known for each of the 36 $\HH$-primes in $\OO(SL_3)$ and are listed in Figure~\ref{fig:H_primes_gens}.  This diagram originally appeared in \cite[Figure~1]{GL1} and described generators of $\HH$-primes in the quantum algebra $\OO_q(SL_3)$, and it was verified in \cite[Chapter~5]{MeThesis} that we can use the same generating sets to obtain the Poisson $\HH$-primes of $\OO(SL_3)$ over arbitrary fields $k$.  Generators are depicted as follows: the dots in the $3 \times 3$ grids represent the indeterminates $Y_{ij}$, $1 \leq i,j\leq 3$ in the natural way; a black dot in position $(i,j)$ means that $Y_{ij}$ is included in the generating set for that ideal, and a square indicates that the corresponding $2\times 2$ minor is in the generating set.  For example, $\hprime{
 \circ &  \hsq  \\
\circ &  &  \\
\bullet & \circ & \circ 
}$
denotes the ideal $\langle Y_{31}, [12|23]\rangle$.

We retain the notation of \cite{GL1} for indexing the 36 $\HH$-primes by elements of $S_3 \times S_3$: for $\omega = (\omega_{+}, \omega_{-}) \in S_3 \times S_3$, the corresponding $\HH$-prime $I_{\omega}$ has the generating set listed in row $\omega_{+}$ and column $\omega_{-}$ of Figure~\ref{fig:H_primes_gens}.  Elements of $S_3$ are written in a truncated form of two-line notation, e.g. $321$ corresponds to the permutation $1 \mapsto 3$, $2 \mapsto 2$, $3 \mapsto 1$.

Our aim in this section is to verify that $R:=\OO(SL_3)$ satisfies the conditions of Proposition~\ref{res:conjecture satisfied if normal generation of primes}, that is:
\begin{itemize}\item[$\bullet$] for each $J \in \HH$-$pspec(R)$, and for every Poisson prime ideal $P$ in $pspec_J(R)$, $P/J$ has a Poisson-normal generating set in $R/J$.\end{itemize}
In order to do this, we will need to understand the centres $\PZ(R_J)$ and obtain explicit generating sets for the Poisson-primitive ideals in $R$.  This formed part of the author's PhD thesis \cite{MeThesis}; since the proofs are computational and long we simply state the results here and refer the interested reader to \cite[Chapter 5]{MeThesis} for the details.

{\renewcommand{\arraystretch}{5}
{\setlength{\tabcolsep}{1.5em}
\begin{sidewaystable}[h]
\begin{tabular}{c|cccccc}
  & 321 & 231 & 312 & 132 & 213 & 123 \\ \hline
321 & \makecell{$[23|12]Y_{13}^{-1}$ \\ $[12|23]Y_{31}^{-1}$} & $[23|12]Y_{13}^{-1}$ & $Y_{12}Y_{23}Y_{31}^{-1}$ &   &   & $Y_{22}[23|12]Y_{31}^{-1}$ \\
231 & $Y_{21}Y_{32}Y_{13}^{-1}$ & \makecell{$[13|23]Y_{21}^{-1}$ \\ $[12|13]Y_{32}^{-1}$} &   & $Y_{11}Y_{23}Y_{32}^{-1}$ & $Y_{12}Y_{33}Y_{21}^{-1}$ &   \\
312 & $[12|23]Y_{31}^{-1}$ &   & \makecell{$[23|13]Y_{12}^{-1}$ \\ $[13|12]Y_{23}^{-1}$} & $Y_{11}Y_{32}Y_{23}^{-1}$ & $Y_{21}Y_{33}Y_{12}^{-1}$ &   \\
132 &   & $Y_{11}Y_{23}Y_{32}^{-1}$ & $Y_{11}Y_{32}Y_{23}^{-1}$ & \makecell{$Y_{11}$ \\ $Y_{23}Y_{32}^{-1}$} &   & $Y_{11}$ \\
213 &   & $Y_{12}Y_{33}Y_{21}^{-1}$ & $Y_{21}Y_{33}Y_{12}^{-1}$ &   & \makecell{$Y_{33}$ \\ $Y_{12}Y_{21}^{-1}$} & $Y_{33}$ \\
123 & $Y_{22}[12|23]Y_{13}^{-1}$ &   &   & $Y_{11}$ & $Y_{33}$ & \makecell{$Y_{11}$ \\ $Y_{22}$}
\end{tabular}\caption{Indeterminates for Poisson centres of $R_{I_{\omega}}$, where $R = \OO(SL_3)$.  Blank entries indicate that the corresponding centre is trivial.}\label{tab:gens of centre}
\end{sidewaystable}
}}

\begin{proposition}\label{res:poisson centre gens for sl3}  
Let $R = \OO(SL_3)$, and let $I_{\omega}$ be a Poisson $\HH$-prime in $R$.  Then the corresponding Poisson centre $\PZ(R_{I_{\omega}})$ is a Laurent polynomial ring in the indeterminates in position $\omega$ of Table~\ref{tab:gens of centre}.  (A blank entry in the table indicates that the corresponding Poisson centre is trivial, i.e. $\PZ(R_{I_{\omega}}) = k$ for that value of $\omega$.)
\end{proposition}

\begin{proof}
\cite[Proposition~5.3.14]{MeThesis}.
\end{proof}
For $\omega \in S_3 \times S_3$, let $\{U_1V_1^{-1}, \dots, U_dV_d^{-1}\}$ denote the generating set for $\PZ(R_{I_{\omega}})$ given in Table~\ref{tab:gens of centre}.  Recall that by Theorem~\ref{res:poisson stratification}, $pprim_{I_{\omega}}(R)$ is homeomorphic to $max(\PZ(R_{I_{\omega}}))$; since $k$ is assumed to be algebraically closed, the maximal ideals of $\PZ(R_{I_{\omega}})$ are
\[M_{\lambda} = \langle U_1V_1^{-1} - \lambda_1, \dots, U_dV_d^{-1} - \lambda_d\rangle, \quad \lambda = (\lambda_1, \dots, \lambda_d) \in (k^{\times})^d,\]
and so the Poisson primitive ideals in $R/I_{\omega}$ are precisely those of the form
\[P_{\lambda} = \big(M_{\lambda}R_{I_{\omega}}\big) \cap R/I_{\omega}.\]
A natural guess at a generating set for $P_{\lambda}$ in $R/I_{\omega}$ would be $\{U_1 - \lambda_1 V_1, \dots, U_d - \lambda_d V_d\}$.  It is quite easy to check that these generators are Poisson-normal in $R/I_{\omega}$ (e.g. see the proof of Theorem~\ref{res:Poisson normal generators for SL3} below) and hence that $\langle U_1 - \lambda_1 V_1, \dots, U_d - \lambda_d V_d \rangle$ is a Poisson ideal, but as yet there is no general approach to proving that this ideal is actually $P_{\lambda}$ itself.  For $\OO(SL_3)$ this is again done using case-by-case analysis in \cite{MeThesis}, and we quote the relevant result below.

{\renewcommand{\arraystretch}{5}
{\setlength{\tabcolsep}{1em}
\small
\begin{sidewaystable}[t!]
\begin{tabular}{c|cccccc}
  & 321 & 231 & 312 & 132 & 213 & 123 \\ \hline
321 & \makecell{$[23|12] - \lambda_1 Y_{13}$ \\ $[12|23] - \lambda_2 Y_{31}$} & $[23|12] - \lambda_1 Y_{13}$ & $Y_{12}Y_{23} - \lambda_1 Y_{31}$ & $0$ & $0$ & $Y_{22}[23|12] - \lambda_1 Y_{31}$ \\
231 & $Y_{21}Y_{32} - \lambda_1 Y_{13}$ & \makecell{$[13|23] - \lambda_1 Y_{21}$ \\ $[12|13] - \lambda_2 Y_{32}$} & $0$ & $Y_{11}Y_{23} - \lambda_1 Y_{32}$ & $Y_{12}Y_{33} - \lambda_1 Y_{21}$ & $0$ \\
312 & $[12|23] - \lambda_1 Y_{31}$ & $0$ & \makecell{$[23|13] - \lambda_1 Y_{12}$ \\ $[13|12] - \lambda_2 Y_{23}$} & $Y_{11}Y_{32} - \lambda_1 Y_{23}$ & $Y_{21}Y_{33} - \lambda_1 Y_{12}$ & $0$ \\
132 & $0$ & $Y_{11}Y_{23} - \lambda_1 Y_{32}$ & $Y_{11}Y_{32} - \lambda_1 Y_{23}$ & \makecell{$Y_{11} - \lambda_1$ \\ $Y_{23} - \lambda_2 Y_{32}$} & $0$ & $Y_{11} - \lambda_1$ \\
213 & $0$ & $Y_{12}Y_{33} - \lambda_1 Y_{21}$ & $Y_{21}Y_{33} - \lambda_1 Y_{12}$ & $0$ & \makecell{$Y_{33} -\lambda_1$ \\  $Y_{12} - \lambda_2 Y_{21}$} & $Y_{33} - \lambda_1$ \\
123 & $Y_{22}[12|23] - \lambda_1 Y_{13}$ & $0$ & $0$ & $Y_{11} -\lambda_1$ & $Y_{33} - \lambda_1$ & \makecell{$Y_{11} - \lambda_1$ \\ $Y_{22} - \lambda_2$}
\end{tabular}
\caption{Generators for Poisson primitive ideals in $\OO(SL_3)/I_{\omega}$; $\lambda_1$, $\lambda_2$ are arbitrary scalars in $k^{\times}$.}\label{tab:gens for primitives}
\end{sidewaystable}
}
}

\begin{theorem} 
Let $R = \OO(SL_3)$, let $I_{\omega}$ be a Poisson $\HH$-prime in $R$, and assume that $k$ is algebraically closed of characteristic 0.  Then generating sets for the Poisson-primitive ideals of $R/I_{\omega}$ are given in position $\omega$ of Table~\ref{tab:gens for primitives}.
\end{theorem}

\begin{proof}
\cite[Theorem~5.4.3]{MeThesis}.
\end{proof}
We will need one further result from \cite{MeThesis}, which is the following:

\begin{proposition}\label{res:SL3 ufd} 
Let $I_{\omega}$ be a Poisson $\HH$-prime in $R$.  Then the quotient $R/I_{\omega}$ is a UFD (in the standard commutative sense).
\end{proposition}

\begin{proof}
\cite[Proposition~5.3.19]{MeThesis}.
\end{proof}

We are now ready to verify that $\varphi_{JK} = f_{JK}^{\circ}\overline{|}g_{JK}^{\circ}$ in $\OO(SL_3)$.

\begin{theorem}\label{res:Poisson normal generators for SL3}
Let $R = \OO(SL_3)$, let $I_{\omega}$ be a Poisson $\HH$-prime in $R$, and assume that $k$ is algebraically closed of characteristic 0.  Then for any $P \in pspec_{I_{\omega}}(R)$, the ideal $P/I_{\omega}$ has a generating set consisting of Poisson-normal elements in $R/I_{\omega}$.
\end{theorem}

\begin{proof}
To simplify the notation, write $J$ for $I_{\omega}$.  From Table~\ref{tab:gens of centre}, we see that $P/J$ has height at most 2.  If $P = J$ then the result is trivial.  When $ht(P/J) =1$, $P/J$ must be principally generated since $R/J$ is a UFD by Proposition~\ref{res:SL3 ufd}; since $P/J$ is a Poisson ideal, this generator must be Poisson-normal.  

This leaves only the case where  $ht(P/J) = 2$, i.e. $P/J$ is maximal in its stratum and hence it is Poisson-primitive.  All that remains is to show that the generating sets given in Table~\ref{tab:gens for primitives} consist of Poisson-normal elements in $R/J$.

Observe that every generator in Table~\ref{tab:gens for primitives} has the form $U -\lambda V$, where $UV^{-1}$ is one of the Poisson-central generators of $\PZ(R_J)$ listed in Table~\ref{tab:gens of centre}.  Since $UV^{-1}$ will also be Poisson-central in $R/J$, we compute that for any $a \in R/J$,
\begin{equation}\label{eq:thing with Poisson centrality}
0 = \{a,UV^{-1}\} = \{a,U\}V^{-1} - \{a,V\}UV^{-2} \implies \{a,U\}V = \{a,V\}U.
\end{equation}
Further, the generators in Table~\ref{tab:gens of centre} have been chosen to ensure that the denominator $V$ is itself always Poisson-normal in $R/J$ (see \cite[Lemma~5.4.1]{MeThesis}).  In other words, for each $a \in R/J$, there is some $a_V \in R/J$ such that $\{a,V\} = a_VV$.  Combining this with \eqref{eq:thing with Poisson centrality}, we have:
\begin{align*}
\{a,U-\lambda V\}V &= \{a,U\}V - \lambda\{a,V\}V \\
&=  \{a,V\}U - \lambda\{a,V\}V\\
&= a_VV(U-\lambda V).
\end{align*}
Since $R/J$ is a domain, we can cancel the $V$ on each side to obtain $\{R/J,U-\lambda V\} \subseteq (U-\lambda V)R/J$, i.e. $U-\lambda V$ is Poisson-normal as required.\end{proof}

\begin{corollary}\label{res:Poisson conjecture holds for SL3}
For any pair of Poisson $\HH$-primes $J \subset K$ in $\OO(SL_3)$, we have
\[\varphi_{JK} = f_{JK}^{\circ}\overline{|}g_{JK}^{\circ},\]
where $f_{JK}^{\circ}\overline{|}g_{JK}^{\circ}$ is defined as in \eqref{eq:definition of overline notation}.
\end{corollary}

\begin{proof}
Combine Proposition~\ref{res:main equality of maps result}, Proposition~\ref{res:conjecture satisfied if normal generation of primes} and Theorem~\ref{res:Poisson normal generators for SL3}.
\end{proof}

\begin{corollary}\label{res:Poisson conjecture holds for GL2}
For any pair of Poisson $\HH$-primes $J \subset K$ in $\OO(GL_2)$, we have
\[\varphi_{JK} = f_{JK}^{\circ}\overline{|}g_{JK}^{\circ}.\]
\end{corollary}

\begin{proof}
Let $a,b,c,d$ denote the standard generators in $\OO(GL_2)$, and $\Delta$ the $2\times 2$ determinant.  There is an isomorphism of Poisson algebras $\OO(SL_3)/J_{132,132} \cong \OO(GL_2)$ which sends $X_{22},X_{23},X_{32},X_{33}$ to $a,b,c,d$ (in that order), and $X_{11}$ to $\Delta^{-1}$.  The result now follows from Theorem~\ref{res:Poisson normal generators for SL3} and Propositions~\ref{res:main equality of maps result}, \ref{res:conjecture satisfied if normal generation of primes}.
\end{proof}


\section{A homeomorphism between $pspec(\OO(SL_3))$ and $spec(\OO_q(SL_3))$}\label{sec:homeomorphism}

In Section~\ref{sec:SL3 poisson stuff} we have seen that we can understand the Zariski topology of $pspec(\OO(SL_3))$ in terms of the commutative algebras $\PZ(R_J)$, $\PZ(R_K)$ and $\PZ_{JK}$, and homomorphisms $f_{JK}$ and $g_{JK}$ between them.  As already noted, this setup is based on the corresponding results for quantum algebras in \cite{GBrown}, which similarly reduces the study of the non-commutative prime spectrum of various quantum algebras to questions about commutative algebras and maps between them.  In this section we will show that for $\OO(SL_3)$ and its uniparameter quantum analogue $\OO_q(SL_3)$ these commutative algebras and homomorphisms are the \textit{same} in each case, and hence that these two topological spaces are homeomorphic.

Throughout this section, fix $q \in k^{\times}$ to be not a root of unity.  We continue to assume that $k$ is algebraically closed of characteristic 0.

We first need to define the algebra of quantum matrices, which quantizes the Poisson structure given in \eqref{eq:2x2 Poisson relations}.  The $2\times 2$ quantum matrix algebra $\OO_q(M_2)$ is defined to be the quotient of the free algebra $k\langle a,b,c,d\rangle$ by the relations
\begin{equation}\label{eq:2x2 quantum minor relations}\begin{gathered}
ab=qba, \qquad ac=qca, \\
bd=qdb, \qquad cd=qdc,\\
bc=cb,\\
ad-da = (q-q^{-1})bc.
\end{gathered}\end{equation}
As in the Poisson case, this extends to a definition of the $m \times p$ quantum matrices by considering the free algebra on $mp$ generators $\{X_{ij}:1 \leq i \leq m, \ 1 \leq j \leq p\}$ and imposing the condition that each set of four variables $\{X_{ij}, X_{il}, X_{kj}, X_{kl}\}$ satisfies a copy of relations \eqref{eq:2x2 quantum minor relations}.  The \textit{quantum determinant} is given by the formula
\begin{equation}\label{eq:quantum det}D_q = \sum_{\pi \in S_m}(-q)^{l(\pi)}X_{1,\pi(1)}X_{2,\pi(2)}\dots X_{n,\pi(n)},\end{equation}
where $l(\pi)$ denotes the length of $\pi \in S_m$, and quantum minors are defined analagously to the Poisson case (but see Notation~\ref{notation for quantum/poisson simultaneously} below).  As we might expect, $D_q$ is a central element in $\OO_q(M_m)$, and we obtain the algebras 
\[\OO_q(SL_m) := \OO_q(M_m)/(D_q-1), \qquad \OO_q(GL_m) := \OO_q(M_m)[D_q^{-1}].\]
Finally, there are rational actions of tori $\HH$ on each of these algebras which are defined in the same way as the Poisson case, see \eqref{eq:H action on M and GL}, \eqref{eq:torus acting on SL}.  We also have a transposition map $\tau: X_{ij} \longrightarrow X_{ji}$ as before, which defines an automorphism on each of the algebras $\OO_q(M_m)$, $\OO_q(GL_m)$, $\OO_q(SL_m)$.

Generating sets for $\HH$-primes in quantum matrices have been studied extensively (e.g. \cite{MR3285618,GLL2,MR2679698}) under various restrictions on $k$ and $q$.  In particular, generating sets for all 230 $\HH$-primes in $\OO_q(M_3)$ are computed in \cite{GL2} for fields of arbitrary characteristic and $q$ not a root of unity.  The $\HH$-primes of $\OO_q(SL_3)$ correspond to exactly those $\HH$-primes of $\OO_q(M_3)$ which do not contain $D_q$; as noted in Section~\ref{sec:SL3 poisson stuff}, we may take the same generating sets for $\HH$-primes in $\OO_q(SL_3)$ and Poisson $\HH$-primes in $\OO(SL_3)$ (up to the obvious modification of replacing quantum minors with minors).  Recall that these generating sets are listed in Figure~\ref{fig:H_primes_gens}.

\begin{notation}\label{notation for quantum/poisson simultaneously}
Since we will often want to treat the quantum and Poisson cases simultaneously in what follows, we make the following conventions.  Generators of $A:=\OO_q(SL_3)$ are denoted by $X_{ij}$ and generators of $R:=\OO(SL_3)$ by $Y_{ij}$; we will use $W_{ij}$ to mean ``$X_{ij}$ in the quantum case, $Y_{ij}$ in the Poisson case''.  Minors and quantum minors will both be denoted by $[I|J]$ and referred to simply as ``minors'', but minors in $\OO_q(SL_3)$ will always be computed using \eqref{eq:quantum det}.  Any references to $\HH$-primes in $\OO(SL_3)$ implicitly mean \textit{Poisson} $\HH$-primes, and ``centre'' should be taken to mean the centre with respect to the non-commutative multiplication or Poisson bracket as appropriate.  Where a formula in the $W_{ij}$ involves terms in $q$, we take $q$ to be a non-zero scalar which is not a root of unity in the quantum case, and $q=1$ in the Poisson case.
\end{notation}

Two types of algebra appear in the definitions of the maps $f_{JK}$, $g_{JK}$: the commutative algebras $\ZZ_{JK}$ and $\PZ_{JK}$ defined in \cite[Definition~3.8]{GBrown} and \eqref{eq:PZ} respectively, and the centres of $\HH$-simple localizations $\ZZ(A_J)$ and $\PZ(R_J)$ appearing in the quantum (resp. Poisson) stratification theorems.  We deal with the latter algebras first.

\begin{definition}\label{def:preservation of notation map} Since we are using the same notation for quantum minors and minors but will sometimes need to emphasise that we are passing from a minor in the quantum world to the corresponding minor in the Poisson world, we make the following definition: let $\theta$ be the ``preservation of notation'' map 
\[\theta: [I|J] \mapsto [I|J],\]
which takes a quantum minor $[I|J]$ and maps it to the corresponding commutative minor $[I|J]$ with the same index sets.  In particular, $\theta(X_{ij}) = Y_{ij}$.
\end{definition}

\begin{remark}
Clearly $\theta$ has no hope of extending to a well-defined map $A \longrightarrow R$; it exists purely as a notational convenience in order to pass from certain generating sets on the quantum side to the corresponding constructions on the Poisson side.  No attempt should be made to expand out a quantum minor and apply $\theta$ as if it were a homomorphism.  We will sometimes need to apply $\theta$ to products and inverses of quantum minors, so we make the following conventions: if $z_1z_2\dots z_n$ is a product of quantum minors, we apply $\theta$ pointwise to the minors \textit{in that order}, i.e.
\[\theta(z_1z_2\dots z_n):= \theta(z_1)\theta(z_2)\dots\theta(z_n).\]
Similarly, if $z$ is an invertible quantum minor we define
\[\theta(z^{-1}) := \theta(z)^{-1}.\]
\end{remark}

\begin{proposition}\label{res:centres of h simple localizations are the same}
Let $J$ be one of the 36 $\HH$-primes appearing in Figure~\ref{fig:H_primes_gens}, viewed as a $\HH$-prime in $A$ or $R$ as appropriate.  Let $A_J$ (resp. $R_J$) denote the localization of $A/J$ (resp $R/J$) at the set of all its regular $\HH$-eigenvectors.  Then $\ZZ(A_J)$ and $\PZ(R_J)$ are both commutative Laurent polynomial rings, there is a presentation
\[\ZZ(A_J) = k[Z_1^{\pm1}, \dots, Z_d^{\pm1}],\]
where each $Z_i$ is a product of quantum minors (and possibly their inverses), and there is an isomorphism of commutative algebras determined by
\[\Theta_J: \ZZ(A_J) \longrightarrow \PZ(R_J): Z_i \mapsto \theta(Z_i), \quad 1 \leq i \leq d.\]
\end{proposition}

\begin{proof}
Generators for the centres $\PZ(R_J)$ are given in Table~\ref{tab:gens of centre}; on the quantum side, the centres are computed for the corresponding algebras $\OO_q(GL_3)_J$ (i.e. $\HH$-simple localizations of the algebras $\OO_q(GL_3)/J$) in \cite[Figure~5]{GL1}.  These can be translated into results for $\OO_q(SL_3)$ via the Levasseur-Stafford isomorphism $\OO_q(GL_3) \cong \OO_q(SL_3)[z^{\pm1}]$ \cite{LS1}; since this isomorphism maps $D_q$ to $z$, and $D_q$ does not belong to any $\HH$-prime of $\OO_q(GL_3)$, it is easily verified that
\[\ZZ\big(\OO_q(GL_3)_J\big) \cong \ZZ\big(A_J[z^{\pm1}]\big) = \ZZ(A_J)[z^{\pm1}],\]
where by a slight abuse of notation we use $J$ to denote both a $\HH$-prime in $\OO_q(GL_3)$ and its projection to $\OO_q(SL_3)$.

In other words, as long as we can find a presentation of $\ZZ(\OO_q(GL_3)_J)$ in which $D_q$ appears as one of the generators, we can then obtain a presentation for $\ZZ(A_J)$ simply by deleting $D_q$ from the list of generators. In all but three cases, namely 
\[\begin{array}{ccc}
\hprime{
 \circ & \bullet & \bullet \\
 \bullet & \circ & \circ \\
 \bullet & \circ & \circ 
}
&
\hprime{
 \circ & \circ & \bullet \\
 \circ & \circ & \bullet \\
 \bullet & \bullet & \circ 
}
&
\hprime{
 \circ & \bullet & \bullet \\
 \bullet & \circ & \bullet \\
 \bullet & \bullet & \circ 
}
\\
(132,132) & (213,213) & (123,123),
\end{array}\]
the result now follows immediately by comparing Table~\ref{tab:gens of centre} with \cite[Figure~5]{GL1}.  (Note that whenever there is a product of quantum minors in \cite[Figure~5]{GL1}, we will always use the order specified in that table.) 

The remaining three cases can be handled by a simple change of variables; for example, when $J = I_{132,132}$ we have $D_q = X_{11}[23|23]$ in $\OO_q(GL_3)/J$.  The generators in position $(132,132)$ of \cite[Figure~5]{GL1} are $\{X_{11},[23|23],X_{23}X_{32}^{-1}\}$ and clearly we can replace the generator $[23|23]$ with $X_{11}[23|23] = D_q$ without changing the algebra; the resulting presentation now matches the one in position (132,132) of Table~\ref{tab:gens of centre}.  The other two cases follow similarly.
\end{proof}

Our main task is to describe the intermediate algebras $\PZ_{JK}$ and $\ZZ_{JK}$.  Let $B$ denote either $\OO(SL_3)$ or $\OO_q(SL_3)$, and let $J \subset K$ be a pair of $\HH$-primes in $B$.  Recall from Lemma~\ref{res:shrinking E_JK set} and \cite[Lemma~3.9]{GBrown} that $\PZ_{JK}$ and $\ZZ_{JK}$ are the (Poisson) centres of localizations of $B/J$ at a multiplicative set $E_{JK}$ (which should also be an Ore set in the non-commutative setting), satisfying the following conditions:
\vspace{0.5em}
\begin{itemize}[leftmargin=1.5cm]
\item[\textbf{(M1)}] $E_{JK}$ consists of regular $\HH$-eigenvectors in $B/J$ which are regular modulo $K/J$;
\item[\textbf{(M2)}] If $L \supseteq J$ is any other $\HH$-prime in $B$ such that $L \not\subseteq K$, then $(L/J) \cap E_{JK} \neq \emptyset.$
\end{itemize}
\vspace{0.5em}

We start by constructing sets $E_{K}$ for each $K$; these will correspond to the case $J = 0$, and the general sets $E_{JK}$ will follow easily as projections of these.

\begin{figure}[h!]
\begin{tabular}{c|cccccc}
$\omega$ & 321 & 231 & 312 & 132 & 213 & 123 \\ \hline
$E_{\omega_{+}}$ & \makecell{$[23|12]$ \\ $W_{31}$ \\ $W_{32}$} & \makecell{$[23|12]$ \\ $W_{32}$} & \makecell{$W_{31}$ \\ $W_{32}$} & $W_{32}$ & $[23|13]$ & $\emptyset$ \\ \hline
$E_{\omega_{-}}$ & \makecell{$[12|23]$ \\ $W_{13}$ \\ $W_{23}$} & \makecell{$W_{13}$ \\ $W_{23}$} & \makecell{$[12|23]$ \\ $W_{23}$} & $W_{23}$ & $[13|23]$ & $\emptyset$
\end{tabular}
\caption{Generators for multiplicative (Ore) sets $E_{K}$.}\label{tab:gens of EEJK}
\end{figure}

\begin{definition}\label{def:E_K with a view to E_JK}For a $\HH$-prime $K$ appearing in position $(\omega_{+},\omega_{-})$ of Figure~\ref{fig:H_primes_gens}, define $E_{K}$ to be the multiplicative set generated in $B$ by
\begin{equation}\label{eq:first def of E_K}E_{\omega_{+}} \cup E_{\omega_{-}} \cup \{[23|23],W_{33}\},\end{equation}
where $E_{\omega_{+}}$ and $E_{\omega_{-}}$ are defined in Figure~\ref{tab:gens of EEJK}.\end{definition}

\begin{remark}\label{rem:Ore sets}
In the quantum setting, we need to know that a multiplicative set satisfies the Ore conditions (e.g. \cite[Chapter 6]{GW1}) before attempting to localize at it.  For all of the localizations we will consider in this paper (and in particular for the sets $E_K$ in Definition~\ref{def:E_K with a view to E_JK}) we get this condition for free thanks to the descriptively named paper ``Every quantum minor generates an Ore set'', \cite{MinorsOreSets}.
\end{remark}

\begin{remark}\label{rem:H primes are difficult}
The commutation relations in $\OO_q(SL_3)$ (resp. the Poisson bracket in $\OO(SL_3)$) mean that we cannot always easily identify which minors belong to a given ideal.  Solving this problem has been the subject of considerable recent work, e.g. \cite{MR2774624,CauchonDD,GLL1}. The choice of generating sets in Figure~\ref{tab:gens of EEJK} in particular is informed by the general theory of Cauchon diagrams and techniques from the study of totally nonnegative matrices, which we do not discuss here.  The interested reader is referred to the survey \cite{LL_TNNsurvey} for more detail on this.

We also highlight the following key results, which hold when $k$ is a field of characterstic zero and $q \in k^{\times}$ is not a root of unity.

\begin{enumerate}
\item A quantum minor $[I|J]$ belongs to a $\HH$-prime $I_{\omega}$ in $\OO_q(SL_3)$ if and only if the corresponding minor $[I|J]$ belongs to the Poisson $\HH$-prime $I_{\omega}$ in $\OO(SL_3)$ \cite[Theorem 4.2]{GLL2}.
\item All prime ideals in $\OO_q(SL_3)$ are completely prime (i.e. prime in the commutative sense) \cite[Corollary II.6.10]{GBbook}.
\end{enumerate}
\end{remark}

\begin{lemma}\label{res:EK and K are disjoint}
For each $\HH$-prime $K$ in $B$ and its corresponding multiplicative set $E_K$, we have $K \cap E_K = \emptyset$.
\end{lemma}
\begin{proof}
This can be verified directly (e.g. by computer) using the results of \cite[Definition 2.6]{GLL1} to contruct a complete list of the minors belonging to a given $\HH$-prime.

For those familiar with Cauchon diagrams and their relationship to $\HH$-primes, there is also a more elegant proof: given a $\HH$-prime $K$ and its Cauchon diagram $C_K$, observe that the generating minors of $E_K$ are chosen to be exactly those minors defined by lacunary sequences (see \cite[Definition~3.1]{LaunoisLenagan1}) starting at the white squares of $C_K$.  The result now follows by combining \cite[Theorem 4.2]{GLL2} and \cite[Proposition 4.2]{LaunoisLenagan1}.  
\end{proof}

We can now verify that the sets $E_K$ have the property (M2) in the case where $J= 0$.

\begin{lemma}\label{res:EE_K kills appropriate primes L}
Let $L$ be a $\HH$-prime in $B$, where $B$ is either $\OO(SL_3)$ or $\OO_q(SL_3)$.  Then for any $\HH$-prime $K$ with $L \not\subseteq K$, we have
\[L \cap E_K \neq \emptyset.\]
\end{lemma}

\begin{proof}
Fix a $\HH$-prime $L \neq 0$.  For any $K$ with $L \not\subseteq K$ there must be at least one element $u \in L \backslash K$; in fact, we can take $u$ to be a generator of $L$, i.e. $u \in \{W_{12}, W_{13},W_{21},W_{23},W_{31},W_{32},[23|12],[12|23]\}$.  We will consider each of these elements in turn and show that if $u \in L$ and $K$ is a $\HH$-prime with $u \not\in K$, then $L \cap E_K \neq \emptyset$.  

First suppose that $u \in\{W_{13},W_{23},W_{31},W_{32}\}$, and let $K$ be a $\HH$-prime with $u \not\in K$.  Comparing Figure~\ref{fig:H_primes_gens} and Figure~\ref{tab:gens of EEJK}, we see that whenever $u$ doesn't appear as a \textit{generator} of $K$, it always appears as a generator of the corresponding set $E_K$; applying Lemma~\ref{res:EK and K are disjoint}, it follows that $u \in E_K$ if and only if $u \not\in K$.  Since $u \in L\backslash K$ by assumption, we have $u \in L \cap E_K \neq \emptyset$ as required.

Now suppose that $u = [23|12] \in L$.  If $K$ is a $\HH$-prime such that $[23|12] \not\in K$, then it must lie somewhere in rows 321 or 231 of Figure~\ref{fig:H_primes_gens} (since these are the only two rows whose $\HH$-primes do not feature $[23|12]$ as either a generator or an obvious corollary of the generators).  But $[23|12]$ appears as a generator for $E_K$ whenever $\omega_+ = 321$ or $231$, and so $u \in L \cap E_K$ as required.  A similar argument handles $u = [12|23]$ (with $\omega_- = 321$ or $312$).

Finally, suppose that $u = W_{21} \in L$.  Using the relation $X_{21}X_{32} - X_{32}X_{21} = (q-q^{-1})X_{22}X_{31}$ in the quantum case and $\{Y_{21},Y_{32}\} = 2Y_{22}Y_{31}$ in the Poisson case, we can conclude that $W_{31} \in L$ as well: this follows from the complete primality of $K$ and the fact that $W_{22} \in K$ would imply $1 \in K$ (e.g. \cite[Lemma 5.1.5]{MeThesis}).  It follows that
\begin{equation}\label{eq:things in L}W_{21}, W_{31}, [23|13], [23|12] \in L.\end{equation}
Now if $K$ is a $\HH$-prime with $W_{21} \not\in K$, then $K$ does not lie in rows 132 or 123 of Figure~\ref{fig:H_primes_gens}.  From Figure~\ref{tab:gens of EEJK}, $E_{K}$ must therefore contain at least one of $W_{31}$, $[23|12]$ and $[23|13]$; it now follows from \eqref{eq:things in L} that $L \cap E_K \neq \emptyset$ by \eqref{eq:things in L}.  A symmetric argument handles the case $u = W_{12}$.
\end{proof}

Since $K \cap E_K = \emptyset$ by Lemma~\ref{res:EK and K are disjoint}, it makes sense to consider the image of the sets $E_K$ modulo another $\HH$-prime $J \subset K$:

\begin{definition}\label{def:E_JK}
Let $J \subset K$ be $\HH$-primes, and let $E_K$ be as in Definition~\ref{def:E_K with a view to E_JK}.  Define $E_{JK}$ to be the image of $E_K$ in $B/J$, where $B$ is either $\OO_q(SL_3)$ or $\OO(SL_3)$ as appropriate.  
\end{definition}

We are now in a position to describe the algebras $\ZZ_{JK}$ and $\PZ_{JK}$ in the setting of $SL_3$. 

\begin{lemma}\label{res:finally EEJK is correct}
Let $J \subset K$ be $\HH$-primes in $B$, where $B$ is either $\OO_q(SL_3)$ or $\OO(SL_3)$.  If $L \in \HH$-$spec(B)$ satisfies $J \subseteq L \not\subseteq K$, then in $B/J$ we have
\[E_{JK} \cap L/J \neq \emptyset.\]
\end{lemma}
\begin{proof}
By Lemma~\ref{res:EE_K kills appropriate primes L}, we have $E_{0K} \cap L \neq \emptyset$ in $B$, i.e. there exists some non-zero $v \in L \cap E_{0K} = L \cap E_{K}$.  This element $v$ continues to be non-zero in $B/J$, since $v \not\in K$ and $J \subset K$.  We therefore have $v + J \in L/J \cap E_{JK} \neq \emptyset$, as required.
\end{proof}

\begin{corollary}\label{res:EE_JK is the correct set}
Let $J \subset K$ be $\HH$-primes in $\OO_q(SL_3)$ or $\OO(SL_3)$, let $\ZZ_{JK}$ be defined as in \cite[Definition~3.8]{GBrown} and $\PZ_{JK}$ defined as in \eqref{eq:PZ}.  Then for the set $E_{JK}$ in Definition~\ref{def:E_JK}, we have:
\[\ZZ_{JK} = \ZZ\Big((\OO_q(SL_3)/J)[E_{JK}^{-1}]\Big) \quad \textrm{ and } \quad \PZ_{JK} = \PZ\Big((\OO(SL_3)/J)[E_{JK}^{-1}]\Big).\]
\end{corollary}

\begin{proof}
In the Poisson case, this follows directly from Lemma \ref{res:shrinking E_JK set} and Lemma \ref{res:finally EEJK is correct}.  For the quantum case, combine \cite[Lemma~3.9]{GBrown}, Lemma~\ref{res:finally EEJK is correct} and \cite{MinorsOreSets}.
\end{proof}

A standard tactic with algebras of this type is to invert certain elements and then make strategic changes of variables in order to simplify the commutation (resp. Poisson bracket) relations.  In particular, our aim will be to reduce as many of the relations as possible to skew-commutation relations (or their Poisson equivalents).

\begin{definition}\label{def:mixed affine spaces}
Let $\mathbf{q} = (a_{ij})$ be an additively skew-symmetric $n \times n$ matrix for some integer $n \geq 2$, and $r$ an integer such that $0 \leq r \leq n$.  Define the skew-commuting $k$-algebra associated to $\mathbf{q}$, $r$ and $n$ by
\[A_{\mathbf{q}}^{r,n}:=k_{\mathbf{q}}[x_1^{\pm1}, \dots, x_r^{\pm1}, x_{r+1}, \dots, x_n] = k\langle x_1^{\pm1}, \dots, x_r^{\pm1}, x_{r+1}, \dots, x_n\rangle/ (x_ix_j = q^{a_{ij}}x_jx_i).\]
When $r=0$, this is a uniparameter quantum affine space on $n$ variables.  We define the corresponding Poisson algebra by
\[P_{\mathbf{q}}^{r,n}:=k[y_1^{\pm1}, \dots, y_r^{\pm1}, y_{r+1},\dots, y_n] \textrm{ with Poisson bracket } \{y_i,y_j\}_{\mathbf{q}} = a_{ij}y_iy_j.\]
In some cases it will be convenient to list the generators of these algebras in a different order (e.g. to list the uninverted variables before the inverted ones); for notational simplicity, we will treat reorderings of the same algebra as equal.  

Given a pair of algebras $A_{\mathbf{q}}^{r,n}$ and $P^{r,n}_{\mathbf{q}'}$, if there is a choice of reordering of the variables in each algebra such that $\mathbf{q} = \mathbf{q}'$, then we will say that the two algebras are \textit{compatible}.
\end{definition}

The centres of these algebras are particularly easy to compute, as we now describe.
\begin{lemma}\label{res:extension of Oh centres result}
Let $\mathbf{q}$ be an additively skew-symmetric $n \times n$ matrix, and $A_{\mathbf{q}}^{r,n}$, $P_{\mathbf{q}}^{r,n}$ the corresponding quantum and Poisson algebras for some $0 \leq r \leq n$ (see Definition~\ref{def:mixed affine spaces}).  Then $\ZZ(A_{\mathbf{q}}^{r,n})$, $\PZ(P_{\mathbf{q}}^{r,n})$ are each generated by their (Poisson) central monomials, and if 
\[\ZZ(A_{\mathbf{q}}^{r,n}) = k[Z_1^{\pm1},\dots, Z_s^{\pm1},Z_{s+1},\dots,Z_t]\]
for some $s, t \geq 0$ then
\[\PZ(P_{\mathbf{q}}^{r,n}) = k[z_1^{\pm1}, \dots, z_s^{\pm1},z_{s+1},\dots, z_t],\]
where if $Z_i = x_1^{t_1}\dots x_m^{t_m}$ then $z_i := y_1^{t_1}\dots y_m^{t_m}$ for each $i$, $1 \leq i \leq t$.
\end{lemma}

\begin{proof}
The first statement is an easy generalisation of \cite[\S9.6]{GoodearlSummary}, and the second statement is proved in \cite[\S2]{OhSymplectic} when $r = n$.  The result for $r < n$ now follows easily by observing
\[\ZZ(A_{\mathbf{q}}^{r,n}) = \ZZ(A_{\mathbf{q}}^{n,n}) \cap A_{\mathbf{q}}^{r,n} \quad \textrm{ and }\quad \PZ(P_{\mathbf{q}}^{r,n}) = \PZ(P_{\mathbf{q}}^{n,n}) \cap P_{\mathbf{q}}^{r,n},\]
and the fact that the (Poisson) centre is generated by the (Poisson) central monomials.
\end{proof}

\begin{remark}\label{rem:computing central monomials}
Given an additively anti-symmetric matrix $\mathbf{q} = (a_{ij})$ defining $A_{\mathbf{q}}^{r,n}$ or $P_{\mathbf{q}}^{r,n}$ as above, it is very easy to compute which monomials are central: in each case, a monomial $w_1^{m_1}w_2^{m_2}\dots w_n^{m_n}$ is central if and only if
\[\sum_{j=1}^n a_{ij}m_j = 0, \qquad \forall i, 1 \leq i \leq n.\]
\end{remark}

Before diving into the case-by-case analysis of Theorem~\ref{res:main theorem on equality of intermediate spaces}, we first illustrate the general approach with some examples.  When performing computations in $\OO_q(SL_3)$ or $\OO(SL_3)$, it is useful to keep in mind the commutation relations for minors listed in \cite[\S1.3]{GL1} (for the quantum case) and \cite[\S5.1.1]{MeThesis} (for the Poisson case).

\begin{example}\label{ex:invert d in gl2}
Let $a,b,c,d$ denote the standard generators in $\OO_q(M_2)$ or $\OO(M_2)$, and $\Delta$ the $2\times 2$ determinant.  If we invert $d$ (recall that this is possible even in the quantum case thanks to the results of \cite{MinorsOreSets}) then $a = (\Delta + qbc)d^{-1}$; this suggests we might try to rewrite the localized algebra in terms of the variables $\{\Delta, b,c,d^{\pm1}\}$.  Indeed, we obtain isomorphisms
\[\OO_q(M_2)[d^{-1}] \cong A_{\mathbf{q}}^{1,4}, \qquad \OO(M_2)[d^{-1}] \cong P_{\mathbf{q}}^{1,4},\]
on the variables $\Delta,b,c,d^{\pm1}$ with respect to the matrix
\[\mathbf{q} = \begin{bmatrix} 0& 0 & 0 & 0 \\ 0 & 0 & 0 & 1 \\ 0 & 0 & 0 & 1 \\ 0 & -1 & -1 & 0 \end{bmatrix}.\]
The fact that we obtain genuine isomorphisms rather than quotients of $A_{\mathbf{q}}^{r,n}$ and $P_{\mathbf{q}}^{r,n}$ follows from a GK dimension argument as in \cite[Lemma 4.3 (b)]{GL1}.
\end{example}

\begin{example}\label{ex:invert stuff in sl3}
Now let $B$ denote $\OO_q(SL_3)$ or $\OO(SL_3)$, and consider what happens when we invert $[23|23]$ and $W_{33}$.  As in Example~\ref{ex:invert d in gl2}, we have $W_{22} = ([23|23] + qW_{23}W_{32})W_{33}^{-1}$ and the relation $D_q=1$ (resp $D = 1$) can be rewritten as
\[W_{11} = (1 + qW_{12}[23|13] - q^2W_{13}[23|12])[23|23]^{-1},\]
i.e. inverting $W_{33}$ and $[23|23]$ allows us to make the change of variables $W_{22} \rightarrow [23|23]$ and eliminate the relation $D_q=1$ in the localized algebra.  This brings the structure of $B\big[[23|23]^{-1},W_{33}^{-1}\big]$ much closer to that of the algebras in Definition~\ref{def:mixed affine spaces}, although currently the relations involving $W_{12}$ or $W_{21}$ are still badly behaved.

If we now also invert $W_{23}$ and $W_{32}$, we can replace the generators $W_{12}$ and $W_{21}$ by $[12|23]$ and $[23|12]$ respectively (using the same idea as Example~\ref{ex:invert d in gl2} above); the advantage of this is that these new elements are (Poisson-)normal.  We have obtained the compatible algebras
\begin{equation}\label{eq:slightly localized SL3 v1}\begin{gathered}\OO_q(SL_3)\big[[23|23]^{-1},X_{23}^{-1},X_{32}^{-1},X_{33}^{-1}\big] \cong A_{\mathbf{q}}^{4,8},\\
\OO(SL_3)\big[[23|23]^{-1},Y_{23}^{-1},Y_{32}^{-1},Y_{33}^{-1}\big] \cong P_{\mathbf{q}}^{4,8},\end{gathered}\end{equation}
on the variables $\{[12|23],W_{13},[23|12],[23|23]^{\pm1},W_{23}^{\pm1},W_{31},W_{32}^{\pm1},W_{33}^{\pm1}\}$.  The entries of the matrix $\mathbf{q}$ can easily be computed from the relations in $\OO_q(SL_3)$ and $\OO(SL_3)$; we leave this as an exercise for the interested reader.

Note that we could instead have chosen to invert the set $\{[23|23],W_{33},[13|23],[23|13]\}$ (corresponding to $E_K$ with $K = I_{213,213}$).  Now we have the identities
\begin{gather*}
W_{12} = ([13|23] + qW_{13}W_{32})W_{33}^{-1}, \qquad W_{21} = ([23|13] + qW_{23}W_{31})W_{33}^{-1}\\
W_{23} = q[13|23]^{-1}\Big([23|23]W_{13} + q^{-2}[12|23]W_{33}\Big)\\
W_{32} = q^{-1}\Big(W_{31}[23|23] + q^2W_{33}[23|12]\Big)[23|13]^{-1}
\end{gather*}
where the final two equalities follow from a Laplace expansion as in \cite[E1.3a]{GL1}.  Some computation with the commutation relations in \cite[\S1.3]{GL1} and Poisson bracket relations in \cite[\S5.1.1]{MeThesis} verifies that we again have compatible algebras
\begin{equation}\label{eq:slightly localized SL3 v2}\OO_q(SL_3)[E_{I_{213,213}}^{-1}] \cong A_{\mathbf{q}'}^{4,8}, \qquad \OO(SL_3)[E_{I_{213,213}}^{-1}] \cong P_{\mathbf{q'}}^{4,8}\end{equation}
this time on the variables $[13|23]^{\pm1},W_{13},[23|13]^{\pm1},[23|23]^{\pm1},[12|23],W_{31},[23|12],W_{33}^{\pm1}$.

Finally, similar computations show that we can combine these two approaches (for example by inverting $W_{23}$ and $[23|13]$, along with the ubiquitous $W_{33}$ and $[23|23]$) and still obtain compatible quantum/Poisson algebras.
\end{example}
\begin{remark}\label{rem:centers of slightly localized SL3}
If we compute the matrix $\mathbf{q}$ for the algebras in \eqref{eq:slightly localized SL3 v1} or \eqref{eq:slightly localized SL3 v2} and apply Remark~\ref{rem:computing central monomials}, we find that the centres of these algebras are $k$ in each case.  This observation will be useful in Theorem~\ref{res:main theorem on equality of intermediate spaces}, since any algebra which embeds into one of the algebras \eqref{eq:slightly localized SL3 v1}-\eqref{eq:slightly localized SL3 v2} will also have trivial centre.
\end{remark}

We will not usually compute the centres $\PZ_{JK}$ and $\ZZ_{JK}$ explicitly; for our purposes, it will suffice to show that for each pair of $\HH$-primes $J \subset K$ the centres $\PZ_{JK}$ and $\ZZ_{JK}$ are isomorphic and that we can move between them in an intuitive way.  We are now in a position to make this idea precise and to prove our first main theorem.

\begin{theorem}\label{res:main theorem on equality of intermediate spaces}
For each pair of (Poisson) $\HH$-primes $J \subset K$, we can find some integers $r,n$ such that
\[\ZZ_{JK} = A_{\mathbf{0}}^{r,n} \quad \textrm{ and } \quad \PZ_{JK} = P_{\mathbf{0}}^{r,n},\]
where $\mathbf{0}$ denotes the $n\times n$ zero matrix.

Further, there is a presentation $\ZZ_{JK} = k[Z_1^{\pm1}, \dots, Z_r^{\pm1},Z_{r+1},\dots, Z_n]$ with each $Z_i$ a product of quantum minors, such that there is an isomorphism of commutative algebras determined by 
\[\Theta_{JK}: \ZZ_{JK} \longrightarrow \PZ_{JK}: Z_i \mapsto \theta(Z_i).\]
(Recall that $\theta$ is the preservation of notation map from Definition~\ref{def:preservation of notation map} which replaces each quantum minor with the corresponding minor.)
\end{theorem}

\begin{proof}
The proof is by case-by-case analysis.  In each case, we fix some $\HH$-prime $J$ and consider all $\HH$-primes $K$ such that $J \subset K$.  Recall that when referring to a $\HH$-prime in terms of the generating sets in Figure~\ref{fig:H_primes_gens}, $I_{\omega_{+},\omega_{-}}$ will denote the $\HH$-prime listed in position $(\omega_{+},\omega_{-})$.

Recall also that there is a transposition automorphism $\tau: W_{ij} \longrightarrow W_{ji}$ defined on both $\OO_q(SL_3)$ and $\OO(SL_3)$.  From Figure~\ref{fig:H_primes_gens} and Figure~\ref{tab:gens of EEJK} we can easily check that $\tau(I_{\omega_{+},\omega_{-}}) = I_{\omega_{-}^{-1},\omega_{+}^{-1}}$ and $\tau(E_{\omega_{+},\omega_{-}}) = E_{\omega_{-}^{-1},\omega_{+}^{-1}}$, so that if $I_{\omega} \subset I_{\omega'}$ then $\tau$ induces an isomorphism of $k$-algebras or Poisson algebras
\[(B/I_{\omega})[E_{I_{\omega},I_{\omega'}}^{-1}] \stackrel{\sim}{\longrightarrow} (B/I_{\omega^{-1}})[E_{I_{\omega^{-1}},I_{\omega'^{-1}}}^{-1}],\]
where we use $\omega^{-1}$ as a shorthand for $(\omega_{-}^{-1}, \omega_{+}^{-1})$.  This symmetry will allow us to reduce the number of examples we need to compute.

\textbf{Case I}. $J = I_{\omega}$, for $\omega$ one of:
\begin{gather*}
(321,132), \ (321,213),\ (231,312), \ (231,123),\ (312,231), \ (312,123) \\
(132,321),\ (132,213),\ (213,321),\ (213,132), \ (123,231),\ (123,312).
\end{gather*}
Recall from \cite[Eqn (3.7)]{GBrown} that $\ZZ_{JK}$ is a subalgebra of $\ZZ(A_J)$, and similarly from Section~\ref{sec:framework} that $\PZ_{JK}$ is a subalgebra of $\PZ(R_J)$.  In each of the 12 cases listed here, we see immediately from \cite[Figure 5]{GL1} and Table~\ref{tab:gens of centre} that $\ZZ(A_J)$ and $\PZ(R_J)$ are both trivial, and hence $\ZZ_{JK} = \PZ_{JK} = k$ for all $K \supset J$ as well.

\textbf{Case II}. $J = I_{\omega}$, for $\omega$ one of:
\[
\begin{array}{cccc}
\hprime{
 \circ & \bullet & \bullet \\
 \bullet & \circ & \bullet \\
 \bullet & \circ & \circ 
}
&
\hprime{
 \circ & \bullet & \bullet \\
 \circ & \circ & \bullet \\
 \bullet & \bullet & \circ 
}
& 
\hprime{
 \circ & \bullet & \bullet \\
 \bullet & \circ & \circ \\
 \bullet & \bullet & \circ 
}
&
\hprime{
  \circ & \circ & \bullet \\
  \bullet & \circ & \bullet \\
  \bullet & \bullet & \circ 
 } \\
(132,123) & (213,123) &  (123,132) & (123,213)
\end{array}\]
In each of these cases, we will see that $\OO_q(SL_3)/J$ and $\OO(SL_3)/J$ are already compatible algebras of form $A_{\mathbf{q}}^{r,n}$, $P_{\mathbf{q}}^{r,n}$, and that localizing at the set $E_{JK}$ just corresponds to inverting some generators of these algebras.  Consider for example $J = I_{132,123}$, where it is easily computed that we have
\begin{gather*}\OO_q(SL_3)/J = A_{\mathbf{q}}^{2,3} = k_{\mathbf{q}}[X_{22}^{\pm1},X_{32},X_{33}^{\pm1}], \\
\OO(SL_3)/J = P_{\mathbf{q}}^{2,3} = k[Y_{22}^{\pm1},Y_{32},Y_{33}^{\pm1}]\end{gather*}
with respect to the matrix
\begin{equation}\label{eq:q in case II}\mathbf{q} = \begin{bmatrix} 0 & 1 & 0 \\ -1 & 0 & 1 \\ 0 & -1 & 0 \end{bmatrix}.\end{equation}
Only one $\HH$-prime lies above $J$, namely $K=I_{123,123}$, the maximal $\HH$-prime.  Since $[23|23] = W_{22}W_{33}$ modulo $J$, we see that $E_{JK} = \{W_{22}^iW_{33}^j:i,j \geq 0\}$ and so $(\OO_q(SL_3)/J)[E_{JK}^{-1}] = A_{\mathbf{q}}^{2,3}$, $(\OO(SL_3)/J)[E_{JK}^{-1}] = P_{\mathbf{q}}^{2,3}$ (with $\mathbf{q}$ as in \eqref{eq:q in case II}).  We can now apply Lemma~\ref{res:extension of Oh centres result} to complete the proof in this case, and the other three cases follow by similar analysis.

\textbf{Case III}. $J = I_{\omega}$, for  $\omega$ one of 
\[\begin{array}{cc}
\hprime{
 \circ & \bullet & \bullet \\
 \bullet & \circ & \circ \\
 \bullet & \circ & \circ 
}
&
\hprime{
 \circ & \circ & \bullet \\
 \circ & \circ & \bullet \\
 \bullet & \bullet & \circ 
}
\\
(132,132)& (213,213)
\end{array}\]
We focus the case $J = I_{132,132}$, since $I_{213,213}$ follows similarly.  Observe that $\OO_q(SL_3)/J \cong \OO_q(GL_2)$, where we write $a,b,c,d, \Delta$ for the standard generators and determinant of $\OO_q(GL_2)$ and the isomorphism is given by $X_{22},X_{23},X_{32},X_{33} \mapsto a,b,c,d$ and $X_{11} \mapsto \Delta^{-1}$.  Similarly, there is an isomorphism of Poisson algebras $\OO(SL_3)/J \cong \OO(GL_2)$.

As in Example~\ref{ex:invert d in gl2}, upon inverting $W_{33}$ and $[23|23]$ we obtain the quantum algebra $A_{\mathbf{q}}^{2,4}$ (resp. Poisson algebra $P_{\mathbf{q}}^{2,4}$, for the same matrix $\mathbf{q}$) on the generators $[23|23]^{\pm1}$, $W_{23}$, $W_{32}$, $W_{33}^{\pm1}$.  Three $\HH$-primes lie above $J$, which we list below with their corresponding multiplicative sets $E_{JK}$:
\begin{align*}
K &= I_{132,123} & E_{JK} &= \{[23|23]^iW_{32}^jW_{33}^k: i, j, k \geq 0\}\\
K &= I_{123,132} & E_{JK} &= \{[23|23]^iW_{23}^jW_{33}^k: i, j, k \geq 0\}\\
K &= I_{123,123} & E_{JK} &= \{[23|23]^iW_{33}^j: i,j \geq 0\}
\end{align*}
In other words, we are simply inverting some of the generators of our algebras $A_{\mathbf{q}}^{2,4}$ and $P_{\mathbf{q}}^{2,4}$ from above, and so the resulting localizations are again of this form (with respect to the same matrix $\mathbf{q}$, possibly larger $r$). The result now follows by Lemma~\ref{res:extension of Oh centres result}.

The case $J = I_{213,213}$ follows by a similar analysis, since $\OO_q(SL_3)/I_{213,213} \cong \OO_q(GL_2)$ and $\OO(SL_3)/I_{213,213} \cong \OO(GL_2)$ as well.

\textbf{Case IV}. $J = I_{\omega}$, for $\omega$ one of:
\[
\begin{array}{cccc}
\hprime{
 \circ & \bullet & \bullet \\
\circ & \circ & \circ \\
\bullet & \circ & \circ 
}
&
\hprime{
 \circ & \circ & \bullet \\
\circ & \circ & \bullet \\
\bullet & \circ & \circ 
}
&
\hprime{
 \circ & \circ & \bullet \\
 \bullet & \circ & \circ \\
 \bullet & \circ & \circ 
}
&
\hprime{
 \circ & \circ & \bullet \\
 \circ & \circ & \circ \\
 \bullet & \bullet & \circ 
}
\\
(231,132) & (231,213) & (132,312) & (213,312) \\
 & & & 
\\
\hprime{
 \circ & \bullet & \bullet \\
\hsq  & \circ \\
 &  & \circ 
}
&
\hprime{
 \circ & \circ & \bullet \\
\hsq  & \bullet \\
 &  & \circ 
}
&
\hprime{
 \circ &  \hsq  \\
 \bullet &  &  \\
 \bullet & \circ & \circ 
}
&
\hprime{
 \circ &  \hsq  \\
 \circ &  &  \\
 \bullet & \bullet & \circ \\
}
\\
(312,132) & (312,213) & (132,231) & (213,231) 
\end{array}
\]
This is very similar to Case III, in that $\OO_q(SL_3)/J$ and $\OO(SL_3)/J$ are easily seen to simplify to the required form upon inverting $[23|23]$ and $W_{33}$.

For example, when $J = I_{231,132}$ we compute as in Case III to obtain the compatible algebras
\begin{equation}\label{eq:localizations in case IV}
\big(\OO_q(SL_3)/J\big)[X_{33}^{-1},[23|23]^{-1}] \cong A_{\mathbf{q}}^{2,5}, \qquad \big(\OO(SL_3)/J\big)[Y_{33}^{-1},[23|23]^{-1}] \cong P_{\mathbf{q}}^{2,5}
\end{equation}
on the generators $[23|23]^{\pm1}, W_{21},W_{23},W_{32},W_{33}^{\pm1}$, with respect to the matrix
\begin{equation*}
\mathbf{q} = \begin{bmatrix} 0&1&0&0&0 \\ -1&0&1&0&0 \\ 0&-1&0&0&1 \\ 0&0&0&0&1 \\ 0&0&-1&-1&0 \end{bmatrix}
\end{equation*}
There are 7 $\HH$-primes lying above $J$, but in each case inverting the remaining elements of $E_{JK}$ corresponds to inverting some of the generators in \eqref{eq:localizations in case IV} (note that $[23|12] \equiv W_{21}W_{32}$ and $[23|13] \equiv W_{21}W_{33}$ modulo $J$), i.e. $B/J[E_{JK}^{-1}]$ is again a quantum/Poisson algebra with respect to the same matrix $\mathbf{q}$.  As in previous cases, the result now follows by Lemma~\ref{res:extension of Oh centres result}, and the cases $\omega \in \{(231,213),(132,312),(213,312)\}$ follow by similar analysis. 

The remaining four $\HH$-primes require slightly more care, since they have a $2\times 2$ minor appearing in their generating set; for example, consider $J:= I_{312,132}$.  Upon inverting $W_{33}$ and $[23|23]$, we can make changes of variables $W_{22} \rightarrow [23|23]$, $W_{21} \rightarrow [23|13]$ as in Example~\ref{ex:invert stuff in sl3}; upon substituting these into the equality $[23|12] = 0$ and simplifying, we obtain $W_{31} = q^{-1}[23|23]^{-1}[23|13]W_{32}$.  In other words, we have obtained compatible algebras
\[\OO_q(SL_3)/J\big[[23|23]^{-1},X_{33}^{-1}\big] = A_{\mathbf{q}}^{2,5}, \qquad \OO(SL_3)/J\big[[23|23]^{-1},Y_{33}^{-1}\big],\]
on the generators $[23|23]^{\pm1},W_{23},W_{32},W_{33}^{\pm1},[23|13]$.  The remaining analysis follows as above.

\textbf{Case V}. $J = I_{\omega}$, for $\omega$ one of
\[
\begin{array}{cc}
\hprime{
 \circ &  \hsq  \\
\circ &  &  \\
\bullet & \circ & \circ \\
}
&
\hprime{
 \circ & \circ & \bullet \\
 \hsq  & \circ \\
 &  & \circ \\
}
\\
(231,231) & (312,312) 
\end{array}\]

Let $J := I_{231,231}$.  Similar to the computations for $I_{312,132}$ in Case IV, we find that
\[\OO_q(SL_3)/J\big[[23|23]^{-1},X_{33}^{-1}\big] \cong A_{\mathbf{q}}^{2,6}, \qquad \OO(SL_3)/J\big[[23|23]^{-1},Y_{33}^{-1}\big] \cong P_{\mathbf{q}}^{2,6},\]
which are compatible quantum/Poisson algebras in the variables 
\begin{equation}\label{eq:variables in case V}[23|23]^{\pm1},W_{23},W_{32},W_{33}^{\pm1},W_{21},[13|23].\end{equation}
There are 15 $\HH$-primes lying strictly above $J$, corresponding to $\omega = (\omega_{+},\omega_{-}) \in \{231,132,213,123\} \times \{231,132,213,123\} \Big\backslash (231,231)$.  By observing that $[23|12] \equiv W_{21}W_{32}$, $[23|13] \equiv W_{21}W_{33}$, and $W_{13} \equiv q^{-1}[23|23]^{-1}[13|23]W_{23}$ modulo $J$, we see that localizing at $E_{JK}$ for each choice of $K \supset J$ once again corresponds to inverting some of the generators in \eqref{eq:variables in case V}.  Now apply Lemma~\ref{res:extension of Oh centres result}.

The case $J:= I_{312,312}$ follows from the symmetry induced by $\tau$, since $\tau(I_{231,231}) = I_{312,312}$.

\textbf{Case VI}. $J = I_{\omega}$, for $\omega$ one of 
\[
\begin{array}{cc}
\hprime{
 \circ & \circ & \circ \\
 \bullet & \circ & \circ \\
 \bullet & \bullet & \circ \\
}
&
\hprime{
 \circ & \bullet & \bullet \\
\circ & \circ & \bullet \\
\circ & \circ & \circ \\
}
\\
(123,321) & (321,123)
\end{array}
\]
Consider $J := I_{123,321}$.  There are only 5 $\HH$-primes lying above $J$, namely $I_{\omega}$ with $\omega \in \{(123,\omega_{-}): \omega_{-} \in S_3 \backslash \{321\}\}$, but this time we do not obtain an algebra of the form $A_{\mathbf{q}}^{r,n}$ (resp. $P_{\mathbf{q}}^{r,n}$) when inverting only $W_{33}$ and $[23|23]$.  Instead, we expand our initial Ore set slightly and compute that
\begin{equation}\label{eq:algebras for case VI}\begin{aligned}
\OO_q(SL_3)/J\big[[23|23]^{-1},X_{33}^{-1},X_{23}^{-1}\big] &= A_{\mathbf{q}}^{3,5} = k_{\mathbf{q}}[[12|23],X_{13},X_{22}^{\pm1},X_{23}^{\pm1},X_{33}^{\pm1}], \\
\OO(SL_3)/J\big[[23|23]^{-1},Y_{33}^{-1},Y_{23}^{-1}\big] &= P_{\mathbf{q}}^{3,5} = k[[12|23],Y_{13},Y_{22}^{\pm1},Y_{23}^{\pm1},Y_{33}^{\pm1}],
\end{aligned}\end{equation}
(note that $[23|23] = W_{22}W_{33}$ mod $J$), and the result now follows as above for the cases $K = I_{123,132},I_{123,312},I_{123,231}$.

When $K = I_{123,213}$, we note that 
\[E_{JK} = \{W_{33}^i[23|23]^j[13|23]^k:i,j,k\geq 0\} = \{W_{33}^iW_{22}^jW_{12}^k: i,j,k \geq 0\},\]
(using $[23|23] \equiv W_{22}W_{33}$ and $[13|23] \equiv W_{12}W_{33}$ mod $J$). Since $W_{12}$ has been inverted, we can make the change of variables $W_{23} \rightarrow [12|23]$ and obtain compatible algebras in the variables $W_{11}, W_{12}^{\pm1}, W_{13},W_{22}^{\pm1}, [12|23],W_{33}^{\pm1}$.

One possible choice of $K$ remains: the maximal $\HH$-prime $I_{123,123}$.  As noted above, the algebra $B' := B/J\big[[23|23]^{-1},W_{33}^{-1}\big]$ (i.e. $B/J[E_{JK}^{-1}]$ for $K = I_{123,123}$) is \textit{not} an algebra of the form $A_{\mathbf{q}}^{r,n}$ (resp. $P_{\mathbf{q}}^{r,n}$), since there is no inversion and change of variables we can make to eliminate the problems caused by $W_{12}$.  We introduce a new tactic: observe that $B'$ embeds in its localization $B'[W_{23}^{-1}]$, which induces an embedding of centres (resp. Poisson centres) as well.  $B'[W_{23}^{-1}]$ is precisely the quantum (resp. Poisson) algebra described in \eqref{eq:algebras for case VI}, and we may compute the centre of this algebra explicitly using the formula in Remark~\ref{rem:computing central monomials}.  The centre (resp. Poisson centre) of $B'[W_{23}^{-1}]$ turns out to be just $k$, and hence we must have $\ZZ(B') = k$, $\PZ(B')=k$ as well.

The corresponding result for $J := I_{321,123}$ now follows via $\tau$.

\textbf{Case VII}. $J = I_{\omega}$, for $\omega$ one of
\[
\begin{array}{cccc}
\hprime{
 \circ & \circ & \circ \\
\circ & \circ & \circ \\
\bullet & \circ & \circ \\
}
&
\hprime{
 \circ & \circ & \bullet \\
\circ & \circ & \circ \\
\circ & \circ & \circ \\
}
&
\hprime{
 \circ & \circ & \circ \\
 \hsq  & \circ \\
 &  & \circ \\
}
&
\hprime{
 \circ & \hsq \\
\circ & & \\
\circ & \circ & \circ \\
}
\\
(231,321) & (321,312) & (312,321) & (321,231)
\end{array}
\]
We begin with the case $J:=I_{231,321}$; the others follow by similar reasoning and/or application of the automorphism $\tau$.  There are many $\HH$-primes which contain $J$: all ideals of the form $I_{\omega}$ with $\omega = (\omega_{+},\omega_{-}) \in \{231,132,213,123\} \times S_3$.  Observe that upon inverting $[23|23]$ and $W_{33}$ we are once again only one step away from obtaining an algebra of the desired form: this is because $W_{21}$ is (Poisson-)normal modulo $J$. 

As in previous cases we can compute that $B/J\big[[23|23]^{-1},W_{33}^{-1},W_{23}^{-1}\big]$ is of the form $A_{\mathbf{q}}^{3,7}$, resp. $P_{\mathbf{q}}^{3,7}$ (for an appropriate $\mathbf{q}$), and this is sufficient to handle all $K = I_{\omega}$ with $\omega \in \{231,132,213,123\} \times \{321,231,312,132\}$ in the same manner as previous cases.  Similarly, if we instead invert $\{W_{33},[23|23],[13|23]\}$, we obtain compatible algebras
\begin{equation}\label{eq:algebras in case VII}\begin{gathered}\OO_q(SL_3)/J\big[[23|23]^{-1},X_{33}^{-1},[13|23]^{-1}\big] = A_{\mathbf{q}}^{3,7} \\ 
\OO(SL_3)/J\big[[23|23]^{-1},Y_{33}^{-1},[13|23]^{-1}\big] = P_{\mathbf{q}}^{3,7},\end{gathered}\end{equation}
on the generators $[13|23]^{\pm1},W_{13},W_{21},[23|23]^{\pm1},[12|23],W_{32},W_{33}^{\pm1}$.  Applying Remark~\ref{rem:computing central monomials}, we compute that the only central monomials of \eqref{eq:algebras in case VII} are powers of $(W_{21}W_{32}W_{13}^{-1})^{\pm1}$, and hence $\ZZ_{JK} = \PZ_{JK} = k$ for $K \in \{I_{132,213},I_{213,213},I_{123,213}\}$, and $\ZZ_{JK},\PZ_{JK} = k[W_{13}W_{21}^{-1}W_{32}^{-1}]$ for $K = I_{231,213}$.

Once again, for $K=I_{\omega_{+},123}$, $\omega_{+} \in \{231,132,213,123\}$ the algebra $B/J[E_{JK}^{-1}]$ is not of a form where we can easily compute its centre, and so we apply the same technique as in Case VI.  There are embeddings $B/J[E_{\omega_{+},123}^{-1}] \hookrightarrow B/J[E_{\omega_{+},213}^{-1}]$, and the centres (resp. Poisson centres) of the latter algebras have been computed in the previous paragraph.  Thus for $K \in \{I_{132,123},I_{213,123},I_{123,123}\}$ we immediately conclude that $\ZZ_{JK} = \PZ_{JK} = k$.  Now let $K = I_{231,123}$, and $K' = I_{231,213}$.  From Figure~\ref{tab:gens of EEJK} we have $E_{JK} \subset E_{JK'}$ and hence $\PZ_{JK} \subseteq \PZ_{JK'} = k[Y_{13}Y_{21}^{-1}Y_{32}^{-1}]$; however, since $Y_{13}Y_{21}^{-1}Y_{32}^{-1} \in \OO(SL_3)/J[E_{JK}^{-1}]$ we must in fact have $\PZ_{JK} = \PZ_{JK'}$.  Similarly, $\ZZ_{JK} = \ZZ_{JK'}$.

This concludes the analysis for $J = I_{231,321}$, and the other three cases follow by similar computations and the symmetry induced by $\tau$.

\textbf{Case VIII}. $J = I_{\omega}$, for $\omega$ one of
\[\begin{array}{cc}
\hprime{
 \circ & \circ & \circ \\
\circ & \circ & \circ \\
\circ & \circ & \circ \\
}
&
\hprime{
 \circ & \bullet & \bullet \\
 \bullet & \circ & \bullet \\
 \bullet & \bullet & \circ \\
}
\\
(321,321) & (123,123)
\end{array}\]
The two remaining $\HH$-primes are the extreme cases, i.e. the minimal and maximal $\HH$-primes.  When $J = I_{123,123}$, there is no $\HH$-prime $K$ such that $J \subsetneq K$, so there is nothing to prove in this case.

Now let $J = I_{321,321} = (0)$.  For all but 11 $\HH$-primes $K$ (i.e. all $K = I_{\omega_{+},\omega_{-}}$ with $\omega_{+} \neq 123$, $\omega_{-} \neq 123$) we have already seen in Example~\ref{ex:invert stuff in sl3} that $\OO_q(SL_3)[E_{JK}^{-1}]$ and $\OO(SL_3)[E_{JK}^{-1}]$ are compatible quantum/Poisson algebras, so the result follows immediately in these cases from Lemma~\ref{res:extension of Oh centres result}.

For five of the remaining $\HH$-primes, namely $K = I_{\omega}$ for 
\[\omega \in \{(132,123),(213,123),(123,123),(123,132),(123,213)\},\]
observe that the localization $B[E_{JK}^{-1}]$ embeds in one of the algebras described in Example~\ref{ex:invert stuff in sl3} (see \eqref{eq:slightly localized SL3 v1}, \eqref{eq:slightly localized SL3 v2}) and so the centre is trivial by Remark~\ref{rem:centers of slightly localized SL3}.

For the final six $\HH$-primes, we proceed as in Cases VI and VII: take $E' = E_{JK} \cup \{W_{32}\}$ if $\omega_{+} = 123$, or $E' = E_{JK}\cup\{W_{23}\}$ if $\omega_{-} = 123$.  Then $B[E_{JK}^{-1}]\hookrightarrow B[E'^{-1}]$, $B[E'^{-1}]$ is an algebra of the form $A_{\mathbf{q}}^{r,n}$ (resp. $P_{\mathbf{q}}^{r,n}$) by Example~\ref{ex:invert stuff in sl3}, and we have equalities
\[\ZZ\big(\OO_q(SL_3)[E_{JK}^{-1}]\big) = \ZZ\big(\OO_q(SL_3)[E'^{-1}]\big), \quad \PZ\big(\OO(SL_3)[E_{JK}^{-1}]\big) = \PZ\big(\OO(SL_3)[E'^{-1}]\big),\]
since $W_{23}^{-1}$, $W_{32}^{-1}$ do not appear in any central monomials of $B[E'^{-1}]$ by Remark~\ref{rem:centers of slightly localized SL3}.

The result now follows.
\end{proof}

We continue to write $A:=\OO_q(SL_3)$ and $R:=\OO(SL_3)$.  For each pair of $\HH$-primes $J \subset K$, we have seen that there are isomorphisms $\Theta_J: \ZZ(A_J) \longrightarrow \PZ(R_J)$ (Proposition~\ref{res:centres of h simple localizations are the same}) and $\Theta_{JK}: \ZZ_{JK} \longrightarrow \PZ_{JK}$ (Theorem~\ref{res:main theorem on equality of intermediate spaces}), where the action of $\Theta_{\bullet}$ in each case is simply to replace each quantum minor in a (specified) generating set with the corresponding commutative minor.  

This allows us to identify each pair of algebras $\ZZ(A_J)$ and $\PZ(R_J)$, and similarly $\ZZ_{JK}$ and $\PZ_{JK}$; this induces a natural homeomorphism between the spectra of each pair of algebras as well, which by an abuse of notation we will also denote by $\Theta_J$ (resp. $\Theta_{JK}$).  The following corollary now follows immediately from the definitions of the maps $f_{JK}$ and $g_{JK}$.

\begin{corollary}\label{res:commutativity of middle square}
Let $J \subset K$ be a pair of $\HH$-primes in $A$ (resp. $R$), and let $Y$ be a closed set in $spec(\ZZ(A_J))$.  Then
\[f_{JK}^{\circ}\overline{|}g_{JK}^{\circ} \circ \Theta_J (Y) = \Theta_K \circ f_{JK}^{\circ}\overline{|}g_{JK}^{\circ}(Y).\]
\end{corollary}

We introduce a few final pieces of notation in order to make the statement of the final theorem precise.  For each $\HH$-prime $J \in \HH$-$spec(A)$, let $h_J^q$ denote the homeomorphism $spec_J(A) \approx spec(\ZZ(A_J))$ from the Stratification Theorem, and similarly for each Poisson $\HH$-prime $J \in \HH$-$pspec(R)$ let $h_J^p$ denote the homeomorphism $pspec_J(R) \approx spec(\PZ(R_J))$.  Since there are now many different maps to keep track of, ranging from homomorphisms to homeomorphisms to maps on closed sets, Figure~\ref{fig:magnificent diagram} may help to highlight the most relevant ones and how they fit together.

\begin{theorem}\label{res:theorem homeomorphism}
Let $A := \OO_q(SL_3)$ and $R:=\OO(SL_3)$, and assume that $k$ is an algebraically closed field of characteristic zero and $q \in k^{\times}$ is not a root of unity.  For each $J \in \HH$-$spec(A)$, define a map
\[\psi_J: spec_J(A) \longrightarrow pspec_J(R): P \mapsto (h_J^p)^{-1} \circ \Theta_J \circ h_J^q(P),\]
and let $\psi$ be the map $spec(A) \longrightarrow pspec(R)$ defined by
\[\psi|_{spec_J(A)} = \psi_J \quad \textrm{ for each } J \in \HH\textrm{-}spec(A).\]  
Then $\psi$ is a bijection, and both $\psi$ and $\psi^{-1}$ preserve inclusions, i.e. $\psi$ is a homeomorphism with respect to the Zariski topology.
\end{theorem}

\begin{proof}
It is clear that $\psi$ is a bijection, since it is patched together from the homeomorphisms $\psi_J$.  To prove that $\psi$ and $\psi^{-1}$ preserve inclusions, it is enough to check that $\psi(V(P)) = V(\psi(P))$ for each $P \in spec(A)$.  So we fix some $P \in spec(A)$, and let $J$ be the $\HH$-prime such that $P \in spec_J(A)$.  Now for any $\HH$-prime $K \supseteq J$, we have
\begin{align*}
\psi\big(V(P)\cap spec_K(A)\big) &= \psi_K \circ \varphi_{JK}\Big(V(P) \cap spec_J(A)\Big)\\
&= \varphi_{JK} \circ \psi_J\Big(V(P) \cap spec_J(A)\Big),
\end{align*}
where the first equality follows from the definition of $\varphi_{JK}$ and the second from the fact that if we view all of the maps in Figure~\ref{fig:magnificent diagram} as maps between spaces of closed sets (so we're forgetting some of the homeomorphism structure of the vertical maps), each square of the diagram commutes (by \cite[Theorem~5.3]{GBrown}, Corollary~\ref{res:commutativity of middle square} and Corollary~\ref{res:Poisson conjecture holds for SL3} respectively).

Now since $\psi_J$ is a homeomorphism $spec_J(A)\longrightarrow pspec_J(R)$, we obtain
\begin{align*}\varphi_{JK} \circ \psi_J\Big(V(P) \cap spec_J(A)\Big) &= \varphi_{JK}\Big(V(\psi(P))\cap pspec_J(R)\Big) \\
&= V(\psi(P)) \cap pspec_K(R),
\end{align*}
and the result follows since $V(P)$ is the disjoint union of the $V(P) \cap spec_K(A)$ for $K \supseteq J$, and similarly $V(\psi(P))$ is the disjoint union of the $V(\psi(P)) \cap pspec_K(A)$.
\end{proof}

\begin{figure}[t]
\begin{tikzpicture}[auto]
\node (1) at (0,0) {$spec_J(A)$};
\node (2) at (6,0) {$spec_K(A)$};
\node (3) at (0,-2) {$spec(\ZZ(A_J))$};
\node (4) at (6,-2) {$spec(\ZZ(A_K))$};
\node (7) at (0,-5) {$spec(\PZ(R_J))$};
\node (8) at (6,-5) {$spec(\PZ(R_K))$};
\node (9) at (0,-7) {$pspec_J(R)$};
\node (10) at (6,-7) {$pspec_K(R)$};
\draw [dashed,->] (1) to node {$\varphi_{JK}$} (2);
\draw [->] (1) to node[swap] {$h_J^q$} (3);
\draw [->] (2) to node {$h_K^q$} (4);
\draw [dashed,->] (3) to node {$f_{JK}^{\circ}\overline{|}g_{JK}^{\circ}$} (4);
\draw [->] (3) to node[swap] {$\Theta_J$} (7);
\draw [->] (4) to node {$\Theta_K$} (8);
\draw [dashed,->] (7) to node[swap] {$f_{JK}^{\circ}\overline{|}g_{JK}^{\circ}$} (8);
\draw [<-] (7) to node[swap] {$h_J^p$} (9);
\draw [<-] (8) to node {$h_K^p$} (10);
\draw [dashed,->] (9) to node[swap] {$\varphi_{JK}$} (10);
\end{tikzpicture}
\caption{Some of the maps between the quantum and Poisson strata corresponding to a pair of $\HH$-primes $J \subset K$.  The unbroken arrows denote homeomorphisms, while a dashed arrow indicates a map from the closed sets of the source to closed sets in the target (both w.r.t. the Zariski topology).}\label{fig:magnificent diagram}
\end{figure}

We may also restrict our attention to just the primitive ideals, where we can describe the action of this homeomorphism explicity as follows.

\begin{corollary}\label{res:homeomorphism sl3 primitives theorem}
There is a homeomorphism
\[\psi: prim(\OO_q(SL_3)) \longrightarrow pprim(\OO(SL_3)),\]
obtained by restricting the homeomorphism of Theorem~\ref{res:theorem homeomorphism} to the space of primitive (resp. Poisson primitive) ideals.  Further, for each primitive ideal $P \in prim(\OO_q(SL_3))$, we can find a generating set
\[P = \langle U_1 - \lambda_d V_1, \dots, U_d - \lambda_d V_d \rangle, \qquad \lambda_1,\dots,\lambda_d \in k^{\times},\]
with each $U_i$ and $V_i$ products of quantum minors (we take the empty product of minors to be 1), such that
\[\psi(P) = \big\langle \theta(U_1) - \lambda_1\theta(V_1), \dots, \theta(U_d) - \lambda_d\theta(V_d)\big\rangle.\]
\end{corollary}

\begin{proof}
By \cite[Lemma~9.4(c)]{GoodearlSummary}, the homeomorphism from Theorem~\ref{res:theorem homeomorphism} restricts to a homeomorphism (which we also denote by $\psi$) on the primitive level.  Generating sets for the primitive (respectively Poisson primitive) ideals of $\OO_q(SL_3)$ (resp. $\OO(SL_3)$) are given in \cite[Figure~7]{GL1} (resp. Table~\ref{tab:gens for primitives}), and the second part of the result now follows easily.
\end{proof}

\begin{corollary}\label{res:homeomorphism theorem gl2}
The topological spaces $spec(\OO_q(GL_2))$ and $pspec(\OO(GL_2))$ are homeomorphic, and this homeomorphism restricts to a homeomorphism $prim(\OO_q(GL_2)) \approx pprim(\OO(GL_2))$.
\end{corollary}

\begin{proof}
As observed in Corollary~\ref{res:Poisson conjecture holds for GL2}, there is an isomorphism of Poisson algebras $\OO(GL_2) \cong \OO(SL_3)/I_{132,132}$, and similarly $\OO_q(GL_2) \cong \OO_q(SL_3)/I_{132,132}$.  The result now follows from Theorem~\ref{res:theorem homeomorphism} and Corollary~\ref{res:homeomorphism sl3 primitives theorem}.
\end{proof}

\textbf{Acknowledgements.} I would like to thank Toby Stafford, Ken Goodearl and St\'ephane Launois for many helpful discussions, comments and explanations, and the referee for their careful reading and suggestions which greatly improved the text.  This research was supported by an EPSRC Doctoral Prize fellowship at the University of Leeds.

\end{document}